\documentclass[12pt,reqno]{amsart}
\usepackage[T1]{fontenc}
\usepackage{amsmath,amssymb,amsthm,mathtools,mathrsfs}
\usepackage{microtype}
\usepackage[colorlinks=true,linkcolor=blue,citecolor=blue,urlcolor=blue]{hyperref}
\numberwithin{equation}{section}
\theoremstyle{plain}
\newtheorem{theorem}{Theorem}[section]
\newtheorem{lemma}[theorem]{Lemma}
\newtheorem{proposition}[theorem]{Proposition}

\newtheorem{conjecture}{Conjecture}
\theoremstyle{definition}
\newtheorem{definition}[theorem]{Definition}
\newtheorem{example}[theorem]{Example}
\newtheorem{remark}[theorem]{Remark}
\newcommand{\M}{\mathcal M}

\newcommand{\B}{\mathcal B}
\newcommand{\F}{\mathcal F}
\newcommand{\K}{\mathcal K}
\newcommand{\T}{\mathcal T}
\newcommand{\APP}{\mathrm{APP}_G}
\newcommand{\Hb}{H_{\mathrm b}}
\title[Intermediate entropies for amenable group actions]{Ergodic measures of intermediate entropies for amenable group actions with an approximate product property}
\author{Wenda Zhang}
\address{College of Mathematics and Statistics,\newline Chongqing Jiaotong University, Chongqing 400074, China}
\email{wendazhang951@aliyun.com}
\author[Xiankun Ren]{Xiankun Ren$^{*}$}
\thanks{Corresponding author: Xiankun Ren.}
\address{College of Mathematics and Statistic, Chongqing University,\newline Chongqing 401331, China}
\email{xkren@cqu.edu.cn}
\date{}
\begin{document}

  \begin{abstract}
We study entropy realization for continuous actions of infinite countable
amenable groups. We introduce an approximate product property for amenable group actions that permits a small proportion of tracing mistakes on prescribed
pairwise disjoint, sufficiently invariant finite sets. We prove that this
property implies entropy-denseness and almost entropy-approximability of
every invariant measure. The construction combines zero-entropy exact
tilings with a finite-block estimate for the complexity of tracing mistakes.
Under asymptotic entropy expansiveness, almost entropy-approximability
upgrades to entropy-approximability. For every $0\leq\alpha<h(X,G)$,
ergodic measures of entropy $\alpha$ then form a residual subset of the
invariant measures whose entropy is at least $\alpha$. In particular,
the set of ergodic measure entropies equals $[0,h(X,G)]$.
%Our hypothesis is implied by specification. It is a prescribed-position
%version of approximate orbit tracing; we distinguish it from the classical
%approximate product property, which allows changes in the positions of orbit blocks.
\end{abstract}
\maketitle

% Replacement introduction for the revised manuscript.
% The integrated manuscript relocates Definition def:entropyapprox here.
% It also includes Example ex:app-not-specification in the preliminaries.
\section{Introduction}
The relation between topological entropy and the structure of invariant
measures is a central theme in ergodic theory. The variational principle
identifies topological entropy with the supremum of measure-theoretic
entropies, but it does not determine which values below this supremum are
attained by ergodic measures. This distinction is essential: strictly
ergodic systems with positive topological entropy show that positive
entropy alone need not produce a nontrivial range of ergodic measure
entropies \cite{beguin2007construction,glasner1994strictly,hahn1967entropy}.
Additional dynamical structure is therefore needed to realize all
intermediate values.

In smooth dynamics, an important source of this structure is Katok's
horseshoe approximation theorem for $C^{1+\alpha}$ surface
diffeomorphisms \cite{katok1980lyapunov}. It connects the entropy of
hyperbolic measures with compact hyperbolic sets and leads to the
realization of intermediate entropies. The corresponding question in
higher dimensions is expressed by the following conjecture. Here
$\M_e(X,f)$ denotes the ergodic invariant probability measures of $f$.

\begin{conjecture}[Katok]
If $f$ is a $C^2$ diffeomorphism of a compact Riemannian manifold $X$, then
\[
[0,h_{\mathrm{top}}(f)) \subset \{h_\mu(f):\mu\in\M_e(X,f)\}.
\]
\end{conjecture}

The inclusion in this conjecture is called the \emph{intermediate entropy
property}. In this paper we study its analogue for amenable group actions.
Throughout, $G$ is an infinite countable discrete amenable group, and
$(X,G)$ is a continuous action on a nonempty compact metric space. We
write $\M(X,G)$ and $\M_e(X,G)$ for its invariant and ergodic invariant
Borel probability measures, respectively, equipped with the weak$^*$
topology. We denote topological and measure-theoretic entropy by
$h(X,G)$ and $h_\mu(X,G)$, and put
\[
\mathcal H(X,G)=\{h_\mu(X,G):\mu\in\M_e(X,G)\}.
\]
Our orbit-tracing hypothesis is an approximate product property, denoted
by $\APP$, on prescribed pairwise disjoint sufficiently invariant finite
subsets of $G$; its precise definition is given in
Definition~\ref{APP}. Our first main result realizes the entire closed
entropy interval.

\begin{theorem}\label{mainthm1}
If $(X,G)$ has $\APP$ and is asymptotically entropy expansive, then
\[
\mathcal H(X,G)=[0,h(X,G)].
\]
\end{theorem}

The intermediate entropy problem has been studied through several
different approaches. Results for skew products and toral automorphisms
were obtained by Sun \cite{sun2010zeroentropy,sun2010measures,sun2012density}.
Other settings include partially hyperbolic systems
\cite{ures2012intrinsic,yang2020nonhyperbolic}, hereditary shifts
\cite{konieczny2018arcwise}, transitive systems with shadowing and an
upper semicontinuous entropy map \cite{li2018properties}, and star flows
\cite{li2020measures}. Ergodic universality provides another route to
entropy realization, as in the work of Quas and Soo and subsequent
developments \cite{quas2016ergodic,burguet2020topological,chandgotia2021borel}.
Related applications concern affine transformations of nilmanifolds and
homogeneous systems \cite{huang2021ergodic,guan2017measures}.

In Theorem \ref{mainthm1}, we require that the system has the approximate product property.  The classical approximate product property was first introduced by Pfister and
Sullivan \cite{pfister2005large}, which permits tracing with a small proportion
of errors and some freedom in the placement of successive orbit blocks. The approximate product property
 is almost the weakest one among the so-called {\itshape specification-like} properties, see \cite{kwietniak2016apanorama}
 for a survey. 
For this property, Sun \cite{sun2025ergodic} proved that asymptotic entropy
expansiveness yields much more than the existence of intermediate
entropies: ergodic measures of each prescribed entropy below topological
entropy are residual in the corresponding entropy superlevel set.
The construction of compact invariant subsystems with controlled entropy
is central to his argument. We follow this approximation approach in the
setting of amenable group actions.

For group actions, Liu, Chen and Zhou established related results under
the approximate product property for $\mathbb Z^d$-actions
\cite{liu2026ergodic} and under specification for amenable group actions
\cite{liu2026katok}. Xu and Zhang \cite{xu2026Isotropy} also have a similar result for SFTs with specification.
Passing to a general amenable group requires a
replacement for the concatenation of intervals. Quasitilings provide
large disjoint pieces on which orbit data can be prescribed, while the
zero-entropy exact tilings of Downarowicz, Huczek and Zhang
\cite{downarowicz2019tilings} control the complexity of their placement.
Our aim is to carry out this construction when tracing is allowed to
fail on a small proportion of each prescribed piece.

The property $\APP$ keeps the locations of all orbit pieces fixed and
requires tracing only on most sites of each sufficiently invariant piece.
Specification implies $\APP$ by removing small boundaries
(Proposition~\ref{prop:spec-implies-app}), and the implication is strict.
Indeed, for every infinite countable amenable group, the shift on
\[
\{0,1\}^G\times
\{b\in\{0,1\}^G:|\{g:b(g)=1\}|\leq1\}
\]
is expansive, has entropy $\log 2$, and satisfies $\APP$ but not
specification; see Example~\ref{ex:app-not-specification}. Thus the
orbit-tracing hypothesis allows examples beyond specification even at
positive entropy. The prescribed positions also distinguish $\APP$
from the classical approximate product property: an irrational circle
rotation has the classical property but fails $\APP$
(Remark~\ref{rem:classical-app}). This distinction is relevant when
comparing our hypotheses with those of \cite{sun2025ergodic}.

To describe the approximation results underlying
Theorem~\ref{mainthm1}, we introduce the following terminology, following
\cite{sun2025ergodic,liu2026katok}. For a nonempty compact invariant set
$\Lambda\subset X$, let $h(\Lambda,G,\gamma)$ denote its topological
entropy calculated at the fixed Bowen scale $\gamma$; see
Section~\ref{sec:prelim} for the precise convention.

\begin{definition}\label{def:entropyapprox}
Let $(X,G)$ be an amenable group action and $\mu\in\M(X,G)$.
\begin{enumerate}
\item We say that $\mu$ is \emph{almost entropy-approximable} if for every
neighborhood $U$ of $\mu$, every $h\in(0,h_\mu(X,G))$, and every
$\varepsilon,\beta>0$, there exist a nonempty compact $G$-invariant set
$\Lambda\subset X$ and $\gamma\in(0,\varepsilon)$ such that
\[
\M(\Lambda,G)\subset U,
\qquad h(\Lambda,G)>h,
\qquad h(\Lambda,G,\gamma)<h+\beta.
\]
\item We say that $\mu$ is \emph{entropy-approximable} if for every
neighborhood $U$ of $\mu$, every $h\in(0,h_\mu(X,G))$ and every $\beta>0$
there exists a nonempty compact $G$-invariant set $\Lambda$ such that
\[
\M(\Lambda,G)\subset U,
\qquad h<h(\Lambda,G)<h+\beta.
\]
\item The system is \emph{entropy-generic} if the following holds for
each $0\leq\alpha<h(X,G)$:
\[
\M_e(X,G,\alpha):=\{\nu\in\M_e(X,G):h_\nu(X,G)=\alpha\}
\]
is residual in
\[
\M^\alpha(X,G):=\{\nu\in\M(X,G):h_\nu(X,G)\ge\alpha\}.
\]
\end{enumerate}
If $h_\mu(X,G)=0$, the first two notions mean that for every neighborhood $U$ of $\mu$ there exists a nonempty compact $G$-invariant set $\Lambda$ with $\mathcal M(\Lambda,G)\subset U$.
\end{definition}

Almost entropy-approximability controls the statistical behavior of
\emph{every} invariant measure supported on the approximating set, while
bounding its entropy at a suitably small scale. Entropy-approximability
provides the corresponding bound for full topological entropy. Keeping
these two levels of control separate allows us to isolate the roles of
orbit tracing and asymptotic entropy expansiveness.

\begin{theorem}\label{mainthm2}
If $(X,G)$ has $\APP$, then every invariant measure is almost
entropy-approximable.
\end{theorem}

No expansiveness assumption is needed in this theorem. The main
construction first treats ergodic measures, using separated families of
orbit blocks whose empirical measures approximate the given measure.
Entropy-denseness under $\APP$ then gives the result for arbitrary
invariant measures. Two quantitative issues arise in assembling these
blocks. The pieces obtained by trimming quasitiles must still satisfy the
invariance threshold required by $\APP$, and the permitted tracing
mistakes must not create an uncontrolled amount of entropy. We address
the latter by recording the locations and values of the mistakes in a
finite alphabet $Q$. For a mistake proportion $0<\delta<1/2$, the
additional exponential rate is bounded by
\[
H_{\mathrm b}(\delta)+\delta\log|Q|,
\qquad
H_{\mathrm b}(\delta)=-\delta\log\delta-(1-\delta)\log(1-\delta).
\]
After the alphabet is fixed, this bound tends to zero as $\delta\to0$.
Together with the zero entropy of the tiling system, it gives the upper
entropy estimate needed for almost entropy-approximability.

Asymptotic entropy expansiveness makes the entropy hidden below small
Bowen scales uniformly negligible. It therefore upgrades the preceding
approximation to full entropy control. It also supplies upper
semicontinuity of the entropy map, which is used in the category argument
for the following stronger conclusion.

\begin{theorem}\label{mainthm3}
If $(X,G)$ has $\APP$ and is asymptotically entropy expansive, then:
\begin{enumerate}
\item Every invariant measure is entropy-approximable.
\item The system $(X,G)$ is entropy-generic. Equivalently, for every
$0\leq\alpha<h(X,G)$, the set $\M_e(X,G,\alpha)$ is residual in
$\M^\alpha(X,G)$.
\end{enumerate}
\end{theorem}

Entropy-genericity is a local statement about the space of invariant
measures. In particular, if $\alpha<h(X,G)$ and
$h_\mu(X,G)\geq\alpha$, then every neighborhood of $\mu$ contains an
ergodic measure of entropy exactly $\alpha$. Consequently, whenever
$h_\mu(X,G)<h(X,G)$, every value in $[0,h_\mu(X,G)]$ can be realized
by an ergodic measure arbitrarily close to $\mu$. The upper endpoint
$h(X,G)$ in Theorem~\ref{mainthm1} follows separately from upper
semicontinuity and ergodic decomposition, which give an ergodic measure
of maximal entropy. Its existence does not assert density of maximizing
ergodic measures near every maximizing invariant measure. If
$h(X,G)=0$, the interval in the definition of entropy-genericity is
empty; nevertheless, entropy-denseness implies that the ergodic
measures, all of entropy zero, are residual in $\M(X,G)$.

The paper is organized as follows. Section~\ref{sec:prelim} introduces
the conventions for amenability, orbit tracing and entropy, and explains
the relation between $\APP$ and specification.
Section~\ref{sec:almost} constructs the approximating invariant sets
and proves Theorem~\ref{mainthm2}. Section~\ref{sec:exact} establishes
entropy-approximability and entropy-genericity under asymptotic entropy
expansiveness, and deduces Theorem~\ref{mainthm1}.
Appendix~\ref{sec:appendix} proves entropy-denseness under $\APP$.

\section{Preliminaries}\label{sec:prelim}
%All logarithms are natural. 
We will write the action on the left, so that
$g(hx)=(gh)x$. Tiles and observation windows are translated on the right.

\subsection{Amenability, boundaries and tilings}
Let $\mathcal{F}(G)$ denote the set of nonempty finite subsets of $G$, and let $e_G$ denote the identity of $G$. A sequence
$(F_n)$ in $\F(G)$ is a \emph{left F\o lner sequence} if
\[
\frac{|gF_n\triangle F_n|}{|F_n|}\longrightarrow0
\quad\text{for every }g\in G.
\]
A countable group is amenable if it admits such a sequence. A F\o lner
sequence is \emph{tempered} if, for some $C<\infty$,
\[
\left|\bigcup_{j<n}F_j^{-1}F_n\right|\leq C|F_n|\quad(n\geq1).
\]
Every F\o lner sequence has a tempered subsequence
\cite[Proposition 1.4]{lindenstrauss2001pointwise}.

For $K,F\in\F(G)$, define
\[
\begin{aligned}
\operatorname{Int}_K F&=\{c:Kc\subset F\},\\
\operatorname{Cl}_K F&=\{c:Kc\cap F\ne\varnothing\},\\
\partial_K F&=\operatorname{Cl}_K F\setminus\operatorname{Int}_K F.
\end{aligned}
\]
We say $F\text{ is }(K,\eta)\text{-invariant}$ if $|\partial_KF|<\eta|F|$.
For fixed $K$, $|\partial_KF_n|/|F_n|\to0$ along every F\o lner
sequence \cite[Lemma 9.4.8]{coornaert2015topological}. In addition,
\begin{equation}\label{eq:right-boundary}
\partial_K(Fc)=(\partial_KF)c.
\end{equation}
Thus right translation preserves our invariance condition. The alternative
convention using $|KF\triangle F|/|F|$ is equivalent after changing the
finite test set and tolerance, but the same numerical tolerance must not
be used interchangeably in the two conventions.

A finite family $(A_j)$ is $\kappa$-disjoint if it admits pairwise disjoint
subsets $B_j\subset A_j$ with $|B_j|\geq(1-\kappa)|A_j|$.
We say that $T_1,\dots,T_t$ \emph{$\kappa$-quasitile} a finite set $S$ if
there are finite center sets $C_i$ such that $T_iC_i\subset S$, the family
$\{T_ic:c\in C_i\}$ is $\kappa$-disjoint for each $i$, the unions
$T_iC_i$ are disjoint for distinct $i$, and
\[
\left|\bigcup_{i=1}^tT_iC_i\right|\geq(1-\kappa)|S|.
\]
Empty center sets are allowed.

Next we claim the filling lemma.
\begin{proposition}[{\cite[Lemma 9.4.14]{coornaert2015topological}}]
\label{oldtiling1}
Let $0<\kappa\leq1/2$. There exists an integer $t_0=t_0(\kappa)\geq1$
such that, for every $t\geq t_0$, the following holds.
If nonempty finite sets $T_1,\ldots,T_t,S\subset G$ satisfy
\[
\begin{aligned}
\frac{|\partial_{T_j}T_i|}{|T_i|}&\leq\kappa^{2t}
&& (1\leq j<i\leq t),\\
\frac{|\partial_{T_j}S|}{|S|}&\leq\kappa^{2t}
&& (1\leq j\leq t),
\end{aligned}
\]
then  $S$  can be $\kappa$-quasitiled by$T_1,\ldots,T_t$ .
\end{proposition}
Here we have written the filling lemma using our boundary notation and
our definition of quasitiling. Unlike the classical Ornstein-Weiss quasitiling theorem in \cite{ornstein1987entropy},  no nesting assumption on the $T_i$ is
needed. The finite combinatorial statement holds for any group; amenability
will be used to choose shapes satisfying its hypotheses.

\begin{remark}\label{disjointquasi}
Shrinking the quasitiles gives pairwise disjoint sets $J_{i,c}c\subset S$
with $J_{i,c}\subset T_i$, $|J_{i,c}|\geq(1-\kappa)|T_i|$, and
\[
\left|\bigsqcup_{i,c}J_{i,c}c\right|\geq(1-\kappa)^2|S|.
\]
Indeed, the sum of the untrimmed tile sizes is at least the size of their
union, and each individual tile loses at most a $\kappa$-proportion.
\end{remark}

The following notions and results are from \cite{downarowicz2019tilings}.

\begin{definition}\label{df:tiling}
	A quasitiling $\T$ is determined by two objects:
\begin{enumerate}
	\item a finite collection $\mathcal{S}(\mathcal{T})$ of finite subsets of $G$ containing the unity $e_G$, called the \textit{shapes}.
	\item a finite collection $\mathcal{C}(\mathcal{T}) = \{C(S): S\in\mathcal{S}(\mathcal{T})\}$ of disjoint subsets of $G$, called \textit{center sets} (for the shapes).
\end{enumerate}
The quasitiling is then the family $\mathcal{T} = \{(S,c): S\in\mathcal{S}(\mathcal{T}), c\in C(S)\}$.
We require that the map $(S,c)\mapsto Sc$ is injective.
Hence, by the \textit{tiles} of $\mathcal{T}$ (denoted by the letter $T$) we will mean either the sets $Sc$ or the pairs $(S,c)$ (i.e., the tiles with defined centers), depending on the context.

A quasitiling $\mathcal{T}$ is called an (exact) \textit{tiling} if it is a partition of $G$, i.e.

\[
G=\bigsqcup_{S\in\mathcal S}\bigsqcup_{c\in C(S)}Sc,
\]
\end{definition}

\begin{definition}\cite[Definition 3.5]{downarowicz2019tilings}
Every quasitiling $\mathcal{T}$ can be represented in a symbolic form, as a point $x_{\mathcal{T}} \in \Lambda^{G}$, with the alphabet $\Lambda = \mathcal{S}(\mathcal{T}) \cup \{0\}$, as follows: $x_{\mathcal{T}}(g) = S \iff g \in C(S)$, $0$ otherwise. Let $X_{\mathcal{T}}$ be the orbit closure of $x_{\mathcal{T}}$ under the shift action. This system is called the \textit{dynamical quasitiling} (generated by $\mathcal{T}$). If $\mathcal{T}$ is a tiling, we obtain a \textit{dynamical tiling}. The entropy of $\T$ is defined as
\[
h(\mathcal T)=h(X_{\mathcal T},G).
\]
\end{definition}
% The center configuration
%$x_{\mathcal T}\in(\mathcal S\cup\{0\})^G$ has value $S$ at each
%$c\in C(S)$ and zero elsewhere. We use the shift action
%\[
%(g x)(u)=x(ug)
%\]
%and write $X_{\mathcal T}=\overline{\{gx_{\mathcal T}:g\in G\}}$ and
%$h(\mathcal T)=h(X_{\mathcal T},G)$.

\begin{lemma}[Zero-entropy exact tilings]\label{newtiling1}
For every $K\in\F(G)$ and $\eta>0$, there is a finite-shape exact
tiling $\mathcal T$ of $G$ such that all its shapes are
$(K,\eta)$-invariant and $h(\mathcal T)=0$.
\end{lemma}
This follows from \cite[Theorem 5.2]{downarowicz2019tilings}, converting
invariance conventions if necessary. The shapes can also be required to
have arbitrarily large cardinality: include a finite set $H$ of large
cardinality in the test set and take its boundary tolerance below $1$.
If $|S|<|H|$, then $\operatorname{Int}_H S=\varnothing$, whereas
$S\subset\partial_H S$ when $e_G\in H$, a contradiction.

\subsection{The approximate product property}
We formulate the tracing condition on prescribed sufficiently invariant
sets. This formulation is the hypothesis used throughout the paper.

Let $\mathscr C(X)$ be the set of all sequences of points of $X$,
\[
\mathscr C(X):=\bigl\{\{x_k\}_{k=1}^\infty:x_k\in X\bigr\},
\]
and let $\mathcal{DF}(G)$ be the set of all sequences of pairwise disjoint nonempty
finite subsets of $G$,
\[
\mathcal{DF}(G):=\left\{\{F_k\}_{k=1}^\infty:
F_k\in\F(G),\ F_i\cap F_j=\varnothing\text{ for }i\ne j\right\}.
\]

\begin{definition}\label{def:tracing}
Let $\{x_k\}_{k\ge1}\in\mathscr C(X)$ and
$\{F_k\}_{k\ge1}\in\mathcal{DF}(G)$.
For $\varepsilon>0$, we say that $\{x_k\}$ is
$(\{F_k\},\varepsilon)$-traced by $z\in X$ if
\[
\rho(sz,sx_k)\le\varepsilon,
\qquad s\in F_k,\quad k\ge1.
\]
For $\delta,\varepsilon>0$, we say that $\{x_k\}$ is
$(\{F_k\},\delta,\varepsilon)$-traced by $z$ if
\[
\left|\{s\in F_k:\rho(sz,sx_k)>\varepsilon\}\right|
\le\delta|F_k|,
\qquad k\ge1.
\]
\end{definition}

\begin{definition}\label{finiteAPP}
  We say that $(X,G)$ has the \emph{finite approximate product property}
  if there exists a map
  \[
  m:(0,1)\times(0,1)\longrightarrow \F(G)\times(0,1)
  \]
  such that the following holds. Write
  \[
  m(\delta,\varepsilon)
  =
  \bigl(K(\delta,\varepsilon),\eta(\delta,\varepsilon)\bigr),\qquad e\in K(\delta,\varepsilon).
  \]
  For every $\delta,\varepsilon\in(0,1)$, every $n\in\mathbb N$,
  every $x_1,\ldots,x_n\in X$, and every pairwise disjoint family
  $F_1,\ldots,F_n\in\F(G)$ such that each $F_k$ is
  $\bigl(K(\delta,\varepsilon),\eta(\delta,\varepsilon)\bigr)$-invariant,
  there exists $z\in X$ such that
  \[
  \left|
  \left\{
  s\in F_k:
  \rho(sz,sx_k)>\varepsilon
  \right\}
  \right|
  \leq
  \delta |F_k|,
  \qquad k=1,\ldots,n.
  \]
\end{definition}

\begin{definition}\label{APP}
We say that $(X,G)$ has the \emph{approximate product property} if there exists a map
\[
m:(0,1)\times(0,1)\longrightarrow\F(G)\times(0,1)
\]
such that the following holds.  Write
\[
m(\delta,\varepsilon)=\bigl(K(\delta,\varepsilon),
\eta(\delta,\varepsilon)\bigr),\qquad e\in K(\delta,\varepsilon).
\]
For every $\delta,\varepsilon\in(0,1)$, every sequence
$\{x_k\}_{k\ge1}\in\mathscr C(X)$ and every
$\{F_k\}_{k\ge1}\in\mathcal{DF}(G)$ such that each $F_k$ is
$\bigl(K(\delta,\varepsilon),\eta(\delta,\varepsilon)\bigr)$-invariant,
there exists $z\in X$ which
$(\{F_k\}_{k\ge1},\delta,\varepsilon)$-traces $\{x_k\}_{k\ge1}$.
\end{definition}

\begin{definition}\label{def:specification}
The system $(X,G)$ has the \emph{specification property} if for every
$\varepsilon>0$ there is a nonempty finite set $K_\varepsilon\subset G$ such that,
whenever $E_1,\dots,E_n\in\F(G)$ satisfy
\[
K_\varepsilon E_i\cap E_j=\varnothing\qquad(i\ne j),
\]
and $x_1,\dots,x_n\in X$, there exists $z\in X$ satisfying
\[
\rho(sz,sx_i)\le\varepsilon,
\qquad s\in E_i,\quad i=1,\dots,n.
\]
\end{definition}

\begin{proposition}\label{finiteAPPimpliesAPP}
  The finite approximate product property is equivalent to the approximate
  product property.
\end{proposition}

\begin{proof}
Since an admissible finite family always can be extended to an infinite disjoint
admissible family by adding right translates of one sufficiently invariant
finite set, the  approximate product property implies the finite version.

Conversely, take $\delta,\varepsilon\in(0,1)$, a sequence
$\{x_k\}_{k\ge1}\in\mathscr C(X)$ and 
$\{F_k\}_{k\ge1}\in\mathcal{DF}(G)$ such that each $F_k$ is
$\bigl(K(\delta,\varepsilon),\eta(\delta,\varepsilon)\bigr)$-invariant.

For each $k\geq 1$, let
\[
C_k=\left\{z:\left|\{s\in F_k:\rho(sz,sx_k)>\varepsilon\}\right|
\leq\delta|F_k|\right\}
\]
and $Y_k = \bigcap\limits _{j=1}^{k}C_k$.
Writing $r_k=\lfloor\delta|F_k|\rfloor$, we have
\[
C_k=\bigcup_{\substack{A\subset F_k\\|A|\geq|F_k|-r_k}}
\bigcap_{s\in A}\{z:\rho(sz,sx_k)\leq\varepsilon\}.
\]
Thus $C_k$ is closed. By  the finite approximate product property, the set $Y_k$ is a closed non-empty subset of $X$. Thus we have $\bigcap_{k=1}^{\infty} Y_k\ne\varnothing$ and $z \in \bigcap\limits _{k=1}^{\infty} Y_k$ can $(\{F_k\}_{k\ge1},\delta,\varepsilon)$-trace $\{x_k\}_{k\ge1}$.

\end{proof}

%\begin{remark}
  %The compactness of $X$ is essential in the above argument, whereas the
  %amenability of $G$ is not used directly. The use of the non-strict tracing
  %condition
  %\[
  %\rho(sz,sx_k)\leq\varepsilon
  %\]
  %on the good coordinates guarantees that the corresponding tracing sets are
  %closed, which makes the finite intersection argument possible.
%\end{remark}

\begin{proposition}\label{prop:spec-implies-app}
The specification property implies the finite approximate product property.
\end{proposition}

\begin{proof}
Fix $\delta,\varepsilon\in(0,1)$, take the specification set
$K_\varepsilon$ with $e_{G}\in K_\varepsilon$, and set $m(\delta,\varepsilon)=(K_\varepsilon,\delta)$.
Fix an admissible finite family $(F_i)_{i=1}^n$ and $x_1,\dots,x_k \in X$.
Put
$E_i=\operatorname{Int}_{K_\varepsilon}F_i$. Then
\[
|F_i\setminus E_i|\leq|\partial_{K_\varepsilon}F_i|<\delta|F_i|,
\qquad K_\varepsilon E_i\cap E_j=\varnothing\quad(i\ne j).
\]
By specification, there is $z$ so that
\[
\rho_{E_i}(z, x_i) \leq \varepsilon, 1\leq i \leq k.
\]
All mistakes on $F_i$ lie in
$F_i\setminus E_i$, as required.
\end{proof}

%We have the following example with the help of Chatgpt 5.6.
We give an example with $\APP$ but without specification.
\begin{example}
\label{ex:app-not-specification}
Let $G$ be any infinite countable amenable group. Define
\[
Y=\{b\in\{0,1\}^G:|\{g\in G:b(g)=1\}|\leq1\},
\qquad X=\{0,1\}^G\times Y.
\]
Equip $X$ with the shift action
\[
(s(a,b))(g)=(a(gs),b(gs)),\qquad s,g\in G.
\]
Then $(X,G)$ has $\APP$ in the prescribed-position sense of
Definition~\ref{APP}, but it does not have specification. Moreover, it is
expansive and has topological entropy $\log 2$.
\end{example}

\begin{proof}
The set $Y$ is closed and invariant: its complement is the union of the
cylinders requiring a $1$ at each of two distinct coordinates.
Thus $X$ is a compact invariant subspace of the full shift over the alphabet
$\{0,1\}^2$.

Enumerate $G=\{g_1,g_2,\dots\}$ with $g_1=e_G$, and use the compatible metric
\[
d((a,b),(a',b'))=
\sum_{j=1}^{\infty}2^{-j}
\mathbf 1_{\{(a(g_j),b(g_j))\ne(a'(g_j),b'(g_j))\}}.
\]
We verify $\APP$ with the boundary convention
$|\partial_KF|<\eta|F|$ used in the manuscript.

Fix $\delta,\varepsilon\in(0,1)$. Choose $N$ such that
$2^{-N}<\varepsilon$, and put $W=\{g_1,\dots,g_N\}$.
Choose a finite set $K\supset W$ satisfying
\[
|K|>\frac{2|W|}{\delta},\qquad \eta=\frac\delta2.
\]
This choice of $(K,\eta)$ depends only on $\delta$ and $\varepsilon$.
Every nonempty $(K,\eta)$-invariant set $F$ satisfies
\begin{equation}\label{eq:example-block-size}
|F|\geq|K|>\frac{2|W|}{\delta}.
\end{equation}
Indeed, if $|F|<|K|$, then $\operatorname{Int}_KF=\varnothing$.
Since $e\in K$, one has $F\subset\operatorname{Cl}_KF=\partial_KF$,
contradicting $\eta<1$.

Let $(F_i)_{i\geq1}$ be any pairwise disjoint family of nonempty
$(K,\eta)$-invariant sets, and let $x_i=(a_i,b_i)\in X$.
Define $z=(a,0^G)\in X$ by
\[
a(g)=
\begin{cases}
a_i(g),&g\in F_i,\\
0,&g\notin\bigcup_iF_i.
\end{cases}
\]
This is well-defined because the $F_i$ are disjoint.
For each $i$, let
\[
\begin{aligned}
B_i^{\rm bd}&=F_i\setminus\operatorname{Int}_KF_i,\\
B_i^{\rm sp}&=\{s\in F_i:b_i(ws)=1\text{ for some }w\in W\}.
\end{aligned}
\]
Then $|B_i^{\rm bd}|\leq|\partial_KF_i|<\eta|F_i|$.
Since $b_i$ has at most one nonzero coordinate, $|B_i^{\rm sp}|\leq|W|$:
if that coordinate is $p_i$, every such $s$ lies in $W^{-1}p_i$.
For $s\in F_i\setminus(B_i^{\rm bd}\cup B_i^{\rm sp})$, the points
$sz$ and $sx_i$ agree at all coordinates in $W$, because
$Ws\subset Ks\subset F_i$ and $b_i$ vanishes on $Ws$.
Therefore $d(sz,sx_i)\leq2^{-N}<\varepsilon$. Using
\eqref{eq:example-block-size}, we obtain
\[
\begin{split}
|\{s\in F_i:d(sz,sx_i)>\varepsilon\}|
&\leq |B_i^{\rm bd}|+|B_i^{\rm sp}|\\
&<\frac\delta2|F_i|+|W|<\delta|F_i|.
\end{split}
\]
This proves the infinite-family version of $\APP$ directly.

To disprove specification, fix $\varepsilon_0=1/4$.
For any proposed finite separation set $L\subset G$, choose
$t\notin L\cup L^{-1}\cup\{e\}$. The sets
$E_1=\{e\}$ and $E_2=\{t\}$ satisfy
$LE_i\cap E_j=\varnothing$ for $i\ne j$.
For $p\in G$, write $b^{(p)}$ for the configuration whose unique $1$
is at $p$. Take
\[
x_1=(0^G,b^{(e)}),\qquad x_2=(0^G,b^{(t)}).
\]
If $z=(a,b)$ traces $x_1$ on $E_1$ and $x_2$ on $E_2$ with tolerance
$\varepsilon_0$, then the first inequality forces $b(e)=1$, and the
second forces $b(t)=1$. Indeed, a disagreement at coordinate $e=g_1$
has distance at least $1/2$. This contradicts $b\in Y$.
Thus specification fails.

The same observation shows expansiveness: distinct configurations differ
at some coordinate $s$, so their $s$-translates have distance at least
$1/2$. Finally, for every nonempty finite $F\subset G$,
\[
|\{x|_F:x\in X\}|=2^{|F|}(|F|+1).
\]
The symbolic pattern-count formula for amenable-group entropy gives
\[
h(X,G)=\lim_{n\to\infty}\frac1{|F_n|}
\log\bigl(2^{|F_n|}(|F_n|+1)\bigr)=\log 2
\]
along any F\o lner sequence $(F_n)$.
\end{proof}

\begin{remark}
The sparse coordinate in this example is zero almost surely for every
invariant probability measure. Indeed, the probabilities of the events
$\{b(g)=1\}$ are equal by invariance, while these events are pairwise
disjoint. Since $G$ is infinite, their common probability is zero.
Consequently every invariant measure on $X$ is supported on
$\{0,1\}^G\times\{0^G\}$. The example therefore proves strictness of
the implication from specification to $\APP$, even among expansive
positive-entropy actions, using nonrecurrent points in the sparse coordinate.
\end{remark}

\begin{remark}\label{rem:classical-app}
The positions of the sets $F_i$ in $\APP$ are prescribed. This differs
from the classical approximate product property, which allows small
transition intervals or shifts of blocks. For example, an irrational
circle rotation has the classical property: for every accuracy, a bounded
number of iterates suffices to move any phase close to any prescribed
phase, after which the rotation traces the entire next orbit segment.
The bounded transition times are a vanishing proportion of sufficiently
long segments.

In contrast, no isometric action on a space with two distinct points has
$\APP$. Choose $x_1,x_2$ with $\rho(x_1,x_2)>2\varepsilon$ and
$\delta<1$, and take two disjoint admissible sets. Since
$\rho(sz,sx_i)=\rho(z,x_i)$, the tracing requirement forces
$\rho(z,x_i)\leq\varepsilon$ for both $i$, which is impossible.
The irrational rotation therefore also shows why $\APP$ cannot be
identified with the classical hypothesis in \cite{sun2025ergodic}.
\end{remark}

Entropy-denseness for the classical approximate product property was
established in \cite{pfister2005large}. We prove the corresponding statement
under $\APP$ in Appendix~\ref{sec:appendix}. The metric $D$ used below is
defined in \eqref{metric1}.

\begin{definition}
Let $\mu\in\M(X,G)$.  We say that $\mu$ is \emph{entropy-approachable by ergodic
measures} if for every $\eta>0$ and every $h<h_\mu(X,G)$ there exists
$\nu\in\M_e(X,G)$ such that
\[
D(\mu,\nu)<\eta,
\qquad h_\nu(X,G)>h.
\]
The system is \emph{entropy-dense} if every invariant measure is
entropy-approachable by ergodic measures.
\end{definition}

\begin{proposition}\label{propositionentropydenseb}
If $(X,G)$ has the approximate product property, then $(X,G)$ is entropy-dense.
\end{proposition}
\begin{proof}
The proof is given in Appendix~\ref{sec:appendix}.
\end{proof}

\subsection{Measures and entropy}

For $\varphi\in C(X)$ and $\mu\in \M(X)$, we set
\[
\langle \varphi,\mu \rangle:=\int_X \varphi\,d\mu.
\]
There exists a countable and separating set of continuous functions
$\{\varphi_1,\varphi_2,\dots\}$ such that $0\leq \varphi_k(x)\leq 1$ and 
\begin{align}
D(\mu,\nu):=\sum_{k=1}^{\infty}\frac{1}{2^k}\big|\langle \varphi_k,\mu\rangle-\langle \varphi_k,\nu\rangle\big| \label{metric1}
\end{align}
defines a metric for the weak$^*$-topology on $\M(X)$. Note that for any $\mu,\nu\in \M(X)$, $D(\mu,\nu)\leq 1$. For $\mu \in \M(X), r>0$,the open ball of radius $r>0$ centered at $\mu$  is 
\[
\mathcal{B}(\mu,r) = \{\nu\in \M(X) : D(\mu,\nu) < r\}.
\]

We use the following compatible metric on $X$, still denoted by $\rho$.
The tracing properties above are unchanged by replacing a compatible metric
on a compact space, by uniform equivalence of compatible metrics:

\begin{align}
  \rho(x,y):=D(\delta_x,\delta_y). \label{metric2}
\end{align}

From the definition of metric $D$, it is straightforward to obtain the following proposition.
\begin{proposition}\label{measure0}
For any $n\geq1$, measures $\mu_1,\dots,\mu_n,\nu_1,\dots,\nu_n\in\M(X)$,
and weights $a_1,\dots,a_n\geq0$ satisfying $\sum_{i=1}^{n}a_i =1$, we have
\[
D(\sum_{i=1}^{n}a_i\mu_i, \sum_{i=1}^{n}a_i\nu_i) \leq \sum_{i=1}^{n}a_i D (\mu_i, \nu_i).
\]
\end{proposition}

\subsection{Topological entropy and asymptotic entropy expansiveness}

%Fix a tempered F\o lner sequence $\K=\{K_n\}_{n\ge1}$ with
%$|K_n|/\log(n+1)\to\infty$. 
Fix a F\o lner sequence $\K=\{K_n\}_{n\ge1}$.
 For a nonempty finite
$F\subset G$ define the Bowen metric
\[
\rho_F(x,y):=\max_{g\in F}\rho(gx,gy).
\]
A set is $(F,\varepsilon)$-separated
if distinct points have Bowen distance strictly greater than $\varepsilon$.
For $Y\subset X$ and $\varepsilon>0$, let $s(Y,F,\varepsilon)$ be the maximal
cardinality of an $(F,\varepsilon)$-separated subset of $Y$, and define
\[
h(Y,G,\varepsilon):=
\limsup_{n\to\infty}\frac1{|K_n|}
\log s(Y,K_n,\varepsilon).
\]
If $Y$ is compact and $G$-invariant, then
\[
h(Y,G):=\lim_{\varepsilon\to0}h(Y,G,\varepsilon).
\]
In particular, $h(X,G)$ is the topological entropy of the action and is
independent of the chosen F\o lner sequence.

To define local entropy, for $A\subset X$, $x\in X$ and
$\delta,\varepsilon>0$, set
\[
B_{K_n}(x,\varepsilon):=
\{y\in X:\rho_{K_n}(x,y)<\varepsilon\}
\]
and let
\[
s(K_n,\delta,\varepsilon,A,x)
\]
be the maximal cardinality of a $(K_n,\delta)$-separated subset of
$A\cap B_{K_n}(x,\varepsilon)$.  Put
\[
s(K_n,\delta,\varepsilon,A):=
\sup_{x\in A}s(K_n,\delta,\varepsilon,A,x)
\]
and
\[
h(A,\varepsilon):=
\lim_{\delta\to0}
\limsup_{n\to\infty}\frac1{|K_n|}
\log s(K_n,\delta,\varepsilon,A).
\]
For $\mu\in\M(X,G)$ define
\[
h_{\mu,\mathrm{loc}}(\varepsilon,G):=
\lim_{\sigma\to1^-}
\inf\{h(A,\varepsilon):A\subset X\text{ compact},\ \mu(A)>\sigma\},
\]
and
\[
h_{\mathrm{loc}}(\varepsilon,G):=
\sup_{\mu\in\M(X,G)}h_{\mu,\mathrm{loc}}(\varepsilon,G).
\]

The following definition of asymptotically entropy expansiveness for amenable group actions is from \cite{ren2016local}.
\begin{definition}
The system $(X,G)$ is
\emph{asymptotically entropy expansive} if there is a
F\o lner sequence with respect to which
\[
\lim_{\varepsilon\to0}h_{\mathrm{loc}}(\varepsilon,G)=0.
\]
When this property is assumed, we take $\K$ from a sequence witnessing
this limit by passing to a tempered subsequence and then, if needed, to a
further subsequence satisfying $|K_n|/\log(n+1)\to\infty$.
These passages do not increase any of the local-entropy quantities.
All finite-scale quantities below refer to this fixed observation sequence.
\end{definition}

As shown in \cite{ren2016local}, asymptotic entropy expansiveness implies
upper semicontinuity of the entropy map
$\mu\mapsto h_\mu(X,G)$.  We shall also use the following local-entropy
estimate.

\begin{proposition}[\cite{ren2016local}]\label{asympexpansive2}
For every $\varepsilon>0$,
\[
h(X,G)\le h(X,G,\varepsilon)+h_{\mathrm{loc}}(\varepsilon,G).
\]
More generally, for every compact $G$-invariant $\Lambda\subset X$,
\[
h(\Lambda,G)
\le h(\Lambda,G,\varepsilon)+h_{\mathrm{loc}}^X(\varepsilon,G).
\]
\end{proposition}

\begin{remark}\label{rem:local-restriction}
The superscript $X$ denotes the ambient system. For a compact invariant
$\Lambda\subset X$, its local entropy is at most the ambient local entropy:
intersect each compact set in the defining infimum with $\Lambda$ and
restrict the centers of the Bowen balls. Thus the subsystem estimate follows
from the estimate applied to $(\Lambda,G)$.
For each fixed scale, $h(X,G,\varepsilon)<\infty$, by coding with a finite
cover of diameter less than $\varepsilon$. Hence the local-entropy estimate
implies $h(X,G)<\infty$ under asymptotic entropy expansiveness.
\end{remark}

\begin{lemma}\label{lem:exhaustive-tempered}
Every infinite countable amenable group $G$ admits an increasing left
tempered F\o lner sequence $(F_n)$ such that
\[
e\in F_1,\qquad \bigcup_{n\geq1}F_n=G,\qquad |F_n|\geq n^2.
\]
In particular, $|F_n|/\log(n+1)\to\infty$.
\end{lemma}
\begin{proof}
We first observe that, for finite sets $L,D\subset G$ with $e\in L$,
any $\eta>0$ and any $M>0$, there exists a finite set $F\supset D$
such that $|F|\geq M$ and $|LF\setminus F|<\eta|F|$.
Indeed, amenability gives a nonempty finite set $A$ satisfying
$|LA\setminus A|<(\eta/2)|A|$.
Since $G$ is infinite, we may take arbitrarily many pairwise disjoint
right translates of $A$. Their union $H$ satisfies
$|LH\setminus H|<(\eta/2)|H|$ and can be chosen so large that
$|H|\geq M$ and $|LD|<(\eta/2)|H|$.
Then $F=H\cup D$ has the required properties, since
\[
|LF\setminus F|
\leq |LH\setminus H|+|LD|
<\eta|H|\leq\eta|F|.
\]

Enumerate $G=\{g_1,g_2,\ldots\}$ with $g_1=e$, and put $F_0=\varnothing$.
Having chosen $F_1,\ldots,F_{n-1}$, set
\[
D_n=F_{n-1}\cup\{g_1,\ldots,g_n\},\qquad
L_n=\{g_1,\ldots,g_n\}\cup\bigcup_{j<n}F_j^{-1}.
\]
Apply the preceding observation to choose $F_n\supset D_n$ with
\[
|F_n|\geq n^2,\qquad |L_nF_n\setminus F_n|<\frac1n|F_n|.
\]
The sequence is increasing and exhaustive, and $e\in F_1$.
For every fixed $g_k$ and every $n\geq k$,
\[
\frac{|g_kF_n\triangle F_n|}{|F_n|}
=\frac{2|g_kF_n\setminus F_n|}{|F_n|}<\frac2n,
\]
so it is a left F\o lner sequence. Moreover,
\[
\left|\bigcup_{j<n}F_j^{-1}F_n\right|
\leq |L_nF_n|
<\left(1+\frac1n\right)|F_n|\leq2|F_n|.
\]
Thus the sequence is tempered with constant $2$.
The required growth condition follows from $|F_n|\geq n^2$.
\end{proof}

We fix an auxiliary sequence $(F_n)$ as in Lemma~\ref{lem:exhaustive-tempered}.
It is used only to select generic orbit blocks and quasitiling shapes.
The observation sequence $(K_n)$ need not be increasing.

For a finite Borel partition $\mathcal P$, write
$\mathcal P^F=\bigvee_{g\in F}g^{-1}\mathcal P$. We use the usual
measure entropy
\[
h_\mu(X,G)=\sup_{\mathcal P}
\lim_{n\to\infty}\frac{H_\mu(\mathcal P^{K_n})}{|K_n|}.
\]
The variational principle, affinity of entropy, and the entropy formula for
ergodic decomposition will be used throughout. In particular,
\[
\begin{aligned}
h(X,G)&=\sup_{\mu\in\M_e(X,G)}h_\mu(X,G),\\
h_\mu(X,G)&=\int_{\M_e(X,G)}h_\nu(X,G)\,d\vartheta_\mu(\nu),
\end{aligned}
\]
where $\vartheta_\mu$ is the ergodic decomposition measure; the integral
may be infinite. These are standard entropy facts for amenable actions
\cite{ornstein1987entropy}. The space $\M(X,G)$ is a nonempty compact
metrizable convex set, and its ergodic measures form a $G_\delta$ subset. For more information about ergodic theory for amenable group actions, readers may refer to \cite{kerr2016}.

\subsection{Empirical measures}
We record the empirical-measure estimates used in the construction.

For $x\in X$ and $F\in\F(G)$, define the empirical measure $\mathcal{E}(x,F)$ by
\[
\mathcal{E}(x,F)=\frac{1}{|F|}\sum_{s\in F}\delta_{sx},
\]
where $\delta_x$ denotes the unit measure concentrated on the point $x$.

Given a set $C\subseteq \M(X)$ and $F\in\F(G)$, let
\[
X_{F,C}=\{x\in X:\;\mathcal{E}(x,F)\in C\}.
\]

\begin{lemma}\label{measures1}
If $\varnothing\ne F'\subset F\in\F(G)$, then, for every $x\in X$,
\[
D(\mathcal E(x,F),\mathcal E(x,F'))
\leq\frac{|F\setminus F'|}{|F|}.
\]
\end{lemma}
\begin{proof}
If $F'=F$ there is nothing to prove. Otherwise, with
$t=|F'|/|F|$, write
$\mathcal E(x,F)=t\mathcal E(x,F')+(1-t)\mathcal E(x,F\setminus F')$
and use convexity and $D\leq1$.
\end{proof}

\begin{lemma}\label{measures2}
If $D(\mathcal E(x,F),\mu)<\eta$ and $\varnothing\ne F'\subset F$, then
\[
D(\mathcal E(x,F'),\mu)<\eta+\frac{|F\setminus F'|}{|F|}.
\]
\end{lemma}
\begin{proof}
Apply Lemma~\ref{measures1} and the triangle inequality.
\end{proof}

\begin{lemma}\label{measures3}
Let $F\in\F(G)$, $F'\subset F$, and $x,y\in X$. If
$\rho(sx,sy)\leq\varepsilon$ for $s\in F'$, then
\[
D(\mathcal E(x,F),\mathcal E(y,F))
\leq\varepsilon+\frac{|F\setminus F'|}{|F|}.
\]
\end{lemma}
\begin{proof}
By convexity and \eqref{metric2}, the left side is bounded by
\[
\frac1{|F|}\sum_{s\in F}\rho(sx,sy).
\]
Use $\rho\leq1$ on $F\setminus F'$.
This argument also covers $F'=\varnothing$.
\end{proof}

\begin{lemma}\label{measures5}
Let $P_1,\dots,P_k\subset F$ be pairwise disjoint nonempty sets whose
union has cardinality greater than $(1-\tau)|F|$. Suppose
$D(\mathcal E(x_i,P_i),\mu)<\eta$ and
\[
\left|\{s\in P_i:\rho(sz,sx_i)>\varepsilon\}\right|
\leq\delta|P_i|\quad(1\leq i\leq k).
\]
Then $D(\mathcal E(z,F),\mu)<\eta+\varepsilon+\delta+\tau$.
\end{lemma}
\begin{proof}
Apply Lemma~\ref{measures3} on each $P_i$, take the convex combination
with weights $|P_i|/|\bigcup_jP_j|$, and apply Lemma~\ref{measures1}
to include the uncovered part of $F$.
\end{proof}

\begin{definition}

Let $F\in\F(G), \delta>0$ and $\varepsilon>0$. A subset $\Gamma \subset X$ is \emph{$(\delta,F,\varepsilon)$-separated} if for $x\neq y \in \Gamma,$
\[
\frac{|\{s\in F\mid \rho(sx,sy)>\varepsilon\}|}{|F|} \geq \delta.
\]
For $C\subset\M(X)$, define
\[
N(C;F,\varepsilon) := s(X_{F,C},F,\varepsilon)
\]
and
\[
N(C;\delta, F,\varepsilon) := \max\{|E|:E\subset X_{F,C}\text{ is }(\delta,F,\varepsilon)\text{-separated}\}.
\]
\end{definition}

With this convention, we have the following proposition.
\begin{proposition}[\cite{ren2023on}, Proposition 3.7]\label{mainseppro1}

Let $(F_n)$ be a F\o lner sequence and let $\mu\in\M_e(X,G)$.
For every finite $h^*<h_\mu(X,G)$, there are $\delta^*,\varepsilon^*>0$
such that, for each neighborhood $C$ of $\mu$ and all sufficiently large
$n$, there is a $(\delta^*,F_n,\varepsilon^*)$-separated set
$\Gamma_n\subset X_{F_n,C}$ satisfying
  \[
  \# \Gamma_n\geq e^{h^*|F_n|}.
  \]
\end{proposition}

\begin{lemma}\label{eslemma1}
Suppose $\Gamma\subset X_{F,\B(\mu,\xi)}$ is
$(\delta^*,F,\varepsilon^*)$-separated. If $J\subset F$ and
$|F\setminus J|<a|F|$ with $0<a<\min\{\delta^*/2,\xi,1\}$, then
$\Gamma$ is $(\delta^*/2,J,\varepsilon^*)$-separated and
$\Gamma\subset X_{J,B(\mu,2\xi)}$.
\end{lemma}
\begin{proof}
At least $(\delta^*-a)|F|>(\delta^*/2)|J|$ separating coordinates
survive. Lemma~\ref{measures2} gives the empirical-measure assertion.
This is the trimming argument used in \cite[Lemma 2.6]{zhang2018topological}.
\end{proof}

\section{Almost entropy-approximability}\label{sec:almost}

In this section we prove Theorem~\ref{mainthm2}.  The proof follows the
zero-entropy tiling strategy of \cite{liu2026katok}, but the approximate product
property introduces an additional family of tracing-error sets.  The essential new
estimate is that these error sets have arbitrarily small exponential complexity.

\begin{lemma}\label{lem:trim-invariance}
Let $K,F\in\F(G)$, $J\subset F$, $0\leq a<1$, and $\zeta\geq0$.
If $|F\setminus J|\leq a|F|$ and $|\partial_KF|\leq\zeta|F|$, then
\[
\frac{|\partial_KJ|}{|J|}\leq\frac{\zeta+|K|a}{1-a}.
\]
\end{lemma}
\begin{proof}
If $c\in\partial_KJ\setminus\partial_KF$, then $Kc\subset F$ and
$Kc  \cap F\setminus J \neq \emptyset$. Consequently,
\[
\partial_KJ\subset\partial_KF\cup
\bigcup_{k\in K}k^{-1}(F\setminus J).
\]
Using $|J|\geq(1-a)|F|>0$, we  have
\[
\frac{|\partial_KJ|}{|J|}\leq\frac{\zeta+|K|a}{1-a}.
\]
\end{proof}

\begin{lemma}\label{lem:mistake-code}
For $0<\delta<1/2$ and every positive integer $N$,
\begin{equation}\label{eq:finite-mistake-count}
\sum_{j=0}^{\lfloor\delta N\rfloor}\binom Nj\leq e^{\Hb(\delta)N}.
\end{equation}
For a nonempty finite alphabet $Q$, the number of choices of at most
$\delta N$ sites and their $Q$-values is at most
$\exp((\Hb(\delta)+\delta\log|Q|)N)$.
\end{lemma}
\begin{proof}
For $j\leq\delta N$, since $\delta/(1-\delta)<1$,
\[
\delta^j(1-\delta)^{N-j}
\geq\delta^{\delta N}(1-\delta)^{(1-\delta)N}
=e^{-\Hb(\delta)N}.
\]
Sum the corresponding terms in the binomial expansion of
$(\delta+1-\delta)^N=1$. Assigning values introduces at most
$|Q|^{\delta N}$ choices.
\end{proof}

\begin{lemma}\label{lem:admissible-subtiles}
Fix $K\in \F(G)$ with $e_G\in K$, $\eta_K>0$, and $0<\kappa<1/4$ such that
\begin{equation}\label{eq:trim-budget}
\frac{3|K|\kappa}{1-3\kappa}<\frac{\eta_K}{2}.
\end{equation}
Given a F{\o}lner sequence $(F_n)$ and a lower bound $L$ on
shape cardinalities and $N\geq 1$, there are indices $n_1<\cdots<n_t$ larger than $N$ 
and a zero-entropy exact tiling $\mathcal T$ with shape family $\mathcal S$
such that $|S|\geq L$ for all $S\in\mathcal S$, and every shape has
pairwise disjoint subtiles $J_{S,i,c}c\subset S$, $c\in C_{S,i}$, with
\begin{equation}\label{eq:good-cover-new}
\begin{gathered}
J_{S,i,c}\subset F_{n_i},\qquad
|J_{S,i,c}|>(1-3\kappa)|F_{n_i}|,\\
\left|\bigsqcup_{i,c}J_{S,i,c}c\right|>(1-4\kappa)|S|.
\end{gathered}
\end{equation}
Every $J_{S,i,c}$ is $(K,\eta_K)$-invariant. Moreover, after the finite
block family has been fixed, $\max_i|F_{n_i}|/|S|$ can be made arbitrarily
small for every final shape $S$.
\end{lemma}
\begin{proof}
Choose $\zeta>0$ with
$(\zeta+3|K|\kappa)/(1-3\kappa)<\eta_K$, and fix
$t\geq t_0(\kappa)$ as in Proposition~\ref{oldtiling1}.
Since $(F_n)$ is increasing and F\o lner, choose
$F_{n_1},\ldots,F_{n_t}$ successively in the given tail such that
\[
N \leq n_1, \qquad|\partial_KF_{n_i}|<\zeta|F_{n_i}|,\qquad
|\partial_{F_{n_j}}F_{n_i}|<\kappa^{2t}|F_{n_i}|
\quad(1\leq j<i\leq t).
\]
At each stage there are only finitely many previously chosen test sets,
so all these inequalities can be imposed simultaneously.
Lemma~\ref{newtiling1} supplies a zero-entropy exact tiling with sufficiently
large, $(F_{n_t},\kappa^{2t})$-invariant shapes. Since
$F_{n_j}\subset F_{n_t}$, we have
$\partial_{F_{n_j}}S\subset\partial_{F_{n_t}}S$ for each shape $S$.
Thus all the hypotheses of Proposition~\ref{oldtiling1} hold.
Quasitile each shape and apply Remark~\ref{disjointquasi}.
The resulting estimates
\[
\begin{gathered}
|J_{S,i,c}|\geq(1-\kappa)|F_{n_i}|>(1-3\kappa)|F_{n_i}|,
\\
\left|\bigsqcup_{i,c}J_{S,i,c}c\right|
\geq(1-\kappa)^2|S|>(1-4\kappa)|S|
\end{gathered}
\]
prove \eqref{eq:good-cover-new}. Lemma~\ref{lem:trim-invariance} gives
the required $(K,\eta_K)$-invariance of every $J_{S,i,c}$.
The additional size requirement follows by choosing the final shapes
large relative to the now fixed block family.
\end{proof}

\begin{lemma}\label{lem:complete-local}
Let $\mathcal T$ be a finite-shape exact tiling with $e\in S$ for every
shape, and put $S_*=\bigcup_{S\in\mathcal S}S$. For every $\theta>0$, there is $m_0\geq 1$ such that for $m \geq m_0$ and every
 $g\in G$, the complete tiles contained
in $K_mg$ cover more than $(1-\theta)|K_m|$ points.
Moreover, if $Z\subset X$ satisfies
\[
D(\mathcal E(z,K_mg),\mu)\leq r
\quad \forall z\in Z,\ \forall  g\in G,\ \forall  m\geq m_0,
\]
then every $\nu\in\M(\overline{GZ},G)$ satisfies $D(\nu,\mu)\leq r$.

\end{lemma}
\begin{proof}
A tile crossing the boundary of $K_mg$ has its center in
$\partial_{S_*}(K_mg)$. Its size is at most $|S_*|$.
The size of the uncovered part is bounded by
\[
|S_*|\,|\partial_{S_*}(K_mg)|=o(|K_m|),
\]
uniformly in $g$, by \eqref{eq:right-boundary}.
For localization, use
$\mathcal E(gz,K_m)=\mathcal E(z,K_mg)$ and continuity to pass the
estimate to $\overline{GZ}$. Invariance of the measure $\nu$ gives
\[
\int\mathcal E(x,K_m)\,d\nu(x)=\nu.
\]
Convexity of $D$ then yields
\[
D(\nu,\mu)\leq\int D(\mathcal E(x,K_m),\mu)\,d\nu(x)\leq r.
\]
\end{proof}

\subsection{The construction for an ergodic measure}

\begin{proposition}\label{mainprop1}
Assume that $(X,G)$ has the approximate product property.  Let
$\mu_0\in\M_e(X,G)$ with $h_{\mu_0}(X.G)>0$ and
\[
0<h_0<h_{\mu_0}(X,G).
\]
For every $\eta_0,\beta_0,\gamma_0>0$ there exist
$\gamma\in(0,\gamma_0)$ and a nonempty compact $G$-invariant set
$\Lambda\subset X$ such that
\begin{enumerate}
\item $\M(\Lambda,G)\subset B(\mu_0,\eta_0)$;
\item $h(\Lambda,G)>h_0$;
\item $h(\Lambda,G,\gamma)<h_0+\beta_0$.
\end{enumerate}
\end{proposition}

\begin{proof}
We give the parameter choices first, so that no previously fixed tracing
or tiling data need to be changed later.

%\medskip\noindent\textbf{Step 1. Parameters and generic labels.}
Put $\eta=\eta_0/20$. Choose $\beta>0$ with
\begin{equation}\label{eq:beta-choice-new}
20\beta<\min\{\beta_0,h_{\mu_0}(X,G)-h_0,h_0\},
\end{equation}
where the middle bound is omitted if $h_{\mu_0}(X,G)=\infty$.
Set $h_1=h_0+5\beta$ and $h^*=h_0+10\beta$.
Proposition~\ref{mainseppro1}, along the auxiliary sequence $(F_n)$,
gives constants $\delta^*,\varepsilon^*>0$ and, for all sufficiently
large $n$, sets
\begin{equation}\label{eq:generic-labels-new}
\Gamma_n^*\subset X_{F_n,B(\mu_0,\eta)},\qquad
|\Gamma_n^*|\geq e^{h^*|F_n|},
\end{equation}
which are $(\delta^*,F_n,\varepsilon^*)$-separated.
Choose
\begin{equation}\label{eq:Delta-choice-new}
0<\varepsilon<\min\{\varepsilon^*/20,\eta,\gamma_0/5\}.
\end{equation}
Let $Q\subset X$ be a finite $\varepsilon/4$-spanning set with the minimum cardinality. 
Next choose $0<\delta<\min\{1/2,\delta^*/10,\eta\}$ with
\begin{equation}\label{eq:delta-choice-new}
\Hb(\delta)+\delta\log \#Q<\beta,
\end{equation}
and obtain $m(\delta,\varepsilon)=(K_{\rm APP},\eta_{\rm APP})$.
Finally choose $0<\kappa<1/4$ and $\theta>0$ such that
\begin{align}
&3\kappa<\min\{\delta^*/2,\eta\},\qquad
\frac{3|K_{\rm APP}|\kappa}{1-3\kappa}<\frac{\eta_{\rm APP}}2,
\label{eq:main-trim-choice}\\
&\theta+4\kappa<\eta,\qquad
(\theta+4\kappa)\log \#Q<\beta,
\label{eq:bad-parameters-new}\\
&h_1(1-\theta)(1-4\kappa)>h_0.
\label{eq:low-parameters-new}
\end{align}
These choices are possible because $\delta$ and the APP data are already
fixed before $\kappa$ is chosen.

%\medskip\noindent\textbf{Step 2. The compact set of tracing points.}
Apply Lemma~\ref{lem:admissible-subtiles} with
$(K,\eta_K)=(K_{\rm APP},\eta_{\rm APP})$. Choose its block tail
large enough that \eqref{eq:generic-labels-new} holds and
\[
e^{\beta(1-3\kappa)|F_n|}>2
\]
throughout that tail. By Lemma~\ref{eslemma1}, each $\Gamma_{n_i}^*$
remains $(\delta^*/2,J_{S,i,c},\varepsilon^*)$-separated, and its
empirical measures on $J_{S,i,c}$ lie in $\B(\mu_0,2\eta)$.
Choose $\Gamma_{S,i,c}\subset\Gamma_{n_i}^*$ of cardinality
$\lceil e^{h_1|J_{S,i,c}|}\rceil$. The tail condition and
$h_1+\beta<h^*$ ensure both its existence and
\begin{equation}\label{eq:label-cardinality-new}
e^{h_1|J_{S,i,c}|}\leq|\Gamma_{S,i,c}|
<e^{(h_1+\beta)|J_{S,i,c}|}.
\end{equation}

Index the tiles by $Sd\in\mathcal T$ and set
\[
\Omega=\prod_{Sd\in\mathcal T}\prod_{i=1}^t
\prod_{c\in C_{S,i}}\Gamma_{S,i,c}.
\]
For $\omega\in\Omega$, apply $\APP$ on the pairwise disjoint blocks
$J_{S,i,c}cd$ to the points $d^{-1}c^{-1}\omega_{S,d,i,c}$.
Right invariance \eqref{eq:right-boundary} makes these blocks admissible.
Using the approximate product property, the set
\[
Z^\#(\omega)=\left\{z:
\begin{array}{l}
|\{u\in J_{S,i,c}:\rho(ucdz,u\omega_{S,d,i,c})>\varepsilon\}|\\
\hfill{}\leq\delta|J_{S,i,c}|\quad\text{for every }S,d,i,c
\end{array}\right\}
\]
is nonempty. The relation $R=\{(z,\omega):z\in Z^\#(\omega)\}$
is closed in the compact space $X\times\Omega$, since each block
condition is a finite union of intersections of closed conditions.
Consequently $Z^\#=\pi_X(R)$ is nonempty and compact. Put
\[
Y= \bigcup_{g\in G} g Z^\# = GZ^\#,\text{ and } \Lambda=\overline Y.
\]
Then $\Lambda$ is a nonempty, compact and $G$-invariant subset of $X$.

%\medskip\noindent\textbf{Step 3. Localization of invariant measures.}
For $z\in Z^\#$ choose an admissible labeling $\omega$ so that $z\in Z^{\#}(w)$.  On a good block
$P=Jcd$, where $J=J_{S,i,c}$ has label $x$, Lemma~\ref{measures3} gives
\[
\begin{aligned}
D(\mathcal E(z,P),\mu_0)
&\leq D(\mathcal E(cdz,J),\mathcal E(x,J))
 +D(\mathcal E(x,J),\mu_0)\\
&<\varepsilon+\delta+2\eta<4\eta.
\end{aligned}
\]
By Lemma~\ref{lem:complete-local}, the complete tiles in $K_mg$
cover more than $(1-\theta)|K_m|$ points, uniformly in $g$, for all
large $m$. Their good subtiles cover more than
$(1-\theta)(1-4\kappa)|K_m|$ points. Convexity and
Lemma~\ref{measures1} therefore yield
\begin{equation}\label{eq:uniform-empirical-localization}
D(\mathcal E(z,K_mg),\mu_0)<4\eta+\theta+4\kappa<5\eta
\end{equation}
uniformly in $z,g$ and sufficiently large $m$.
By Lemma~\ref{lem:complete-local}, we have
$D(\nu,\mu_0)\leq5\eta<\eta_0$ for every
$\nu\in\M(\Lambda,G)$. Thus we  prove  item (1).

%\medskip\noindent\textbf{Step 4. Entropy lower bound.}
Let $\Lambda_{K_m}=\{Sd\in\mathcal T:Sd\subset K_m\}$. We call each such $Sd$ good tile of $K_m$. On the good
subtiles of these complete tiles, choose all labels independently, keeping
a fixed reference label on every other block. For each such labeling
$\xi$, choose a tracing point $z_\xi\in Z^\#$.
If $\xi\ne\xi'$, they differ on some good block $Jcd$. Their distinct
labels $x,x'$ are separated on at least $(\delta^*/2)|J|$ coordinates.
The two tracing points together have at most $2\delta|J|$ mistakes there.
Since $2\delta<\delta^*/2$, a separating non-mistake coordinate $u\in J$
remains, and
\[
\rho(ucdz_\xi,ucdz_{\xi'})>
\varepsilon^*-2\varepsilon>\varepsilon.
\]
As $ucd\in K_m$, the points $z_\xi$ form an
$(K_m,\varepsilon)$-separated set. Thus
\[
\begin{split}
\log s(\Lambda,K_m,\varepsilon)
&\geq\sum_{Sd\in\mathcal I_m}\sum_{i,c}\log|\Gamma_{S,i,c}|\\
&\geq h_1\sum_{Sd\in\mathcal I_m}\sum_{i,c}|J_{S,i,c}|\\
&>h_1(1-\theta)(1-4\kappa)|K_m|
\end{split}
\]
for every sufficiently large $m$. Consequently,
\[
h(\Lambda,G)\geq h(\Lambda,G,\varepsilon)
\geq h_1(1-\theta)(1-4\kappa)>h_0.
\]
Thus we prove item (2).

%\medskip\noindent\textbf{Step 5. Finite-scale entropy upper bound.}
Set $\gamma=5\varepsilon<\gamma_0$. For each $y\in Y$, fix a
representation $y=g_yz_y$ and an admissible labeling for $z_y$.
We will encode its orbit on $K_m$ in four parts.

%\smallskip\emph{Tiling patterns.} 
The ideas in this part are from \cite{liu2026ergodic}.
Recall that $\mathcal S$ is the finite family of shapes of the fixed
tiling $\mathcal T$. Put
\[
S_*=\bigcup_{S\in\mathcal S}S,
\qquad L_m=K_m\cup S_*^{-1}K_m.
\]
The set $S_*$ is finite and independent of $m$. Since $(K_m)$ is a
 F\o lner sequence, we have
\[
0\leq\frac{|L_m\setminus K_m|}{|K_m|}
\leq\sum_{s\in S_*}\frac{|s^{-1}K_m\setminus K_m|}{|K_m|}
\longrightarrow0.
\]
Consequently $|L_m|/|K_m|\to1$. Moreover, for every fixed $a\in G$,
\[
\begin{aligned}
	|aL_m\triangle L_m|
	&\leq |aL_m\triangle aK_m|
	+|aK_m\triangle K_m|+|K_m\triangle L_m|\\
	&=2|L_m\setminus K_m|+|aK_m\triangle K_m|.
\end{aligned}
\]
Dividing by $|L_m|\geq|K_m|$ shows that $(L_m)$ is also a 
F\o lner sequence.

Let $\mathcal A=\mathcal S\cup\{0\}$ be the alphabet of the center
configuration $x_{\mathcal T}$. For $g\in G$, define
\[
p_m(g)\in\mathcal A^{L_m},\qquad
p_m(g)(a)=x_{\mathcal T}(ag)\quad(a\in L_m).
\]
Thus all patterns are indexed by the same finite set $L_m$, even though
they describe the different windows $L_mg$ in the original tiling.
The value $p_m(g)(a)=S\in\mathcal S$ means that $Sag$ is a tile of
$\mathcal T$; the value $0$ means that $ag$ is not a tile center.

We enlarge $K_m$ to $L_m$ because the center of a tile meeting $K_mg$
need not belong to $K_mg$. Indeed, if $Sd\cap K_mg\ne\varnothing$,
then $sd=kg$ for some $s\in S$ and $k\in K_m$, and hence
\[
dg^{-1}=s^{-1}k\in S^{-1}K_m\subset L_m.
\]
Therefore the center and shape of every such tile are recorded by
$p_m(g)$. More precisely, after right multiplication by $g^{-1}$, the
tiles meeting the window are exactly
\[
\{Sa:a\in L_m,\ p_m(g)(a)=S\in\mathcal S,
\ Sa\cap K_m\ne\varnothing\}.
\]
In particular, the same pattern also identifies the complete tiles,
which are precisely those satisfying $Sa\subset K_m$.

For each shape $S$, the sets $J_{S,i,c}$ and the centers
$c\in C_{S,i}$ were fixed before the labeling was chosen. Thus, if
$p_m(g)(a)=S$, the good subtiles in the tile $Sag$ are
$J_{S,i,c}cag$, and their coordinates relative to $K_mg$ are
$J_{S,i,c}ca$. It follows that $p_m(g)$ determines all these relative
subtile locations, including which lie in complete tiles. In particular,
two translations giving the same pattern give exactly the same
geometric domains for the subsequent parts of the code. The pattern
records these domains, while the orbit labels on them will be recorded
separately.

It remains to bound the number of possible patterns. By the shift
convention $(gx)(a)=x(ag)$,
\[
p_m(g)=(gx_{\mathcal T})|_{L_m}.
\]
Let $X_{\mathcal T}=\overline{Gx_{\mathcal T}}$ and set
$N_m=\#\{p_m(g):g\in G\}$. Since the alphabet $\mathcal{A}$ is finite, every
pattern on $L_m$ appearing in $X_{\mathcal T}$ already appears in
some translate of $x_{\mathcal T}$: the cylinder specifying that
finite pattern is open and meets the dense orbit $Gx_{\mathcal T}$.
Hence
\[
N_m=\#\{x|_{L_m}:x\in X_{\mathcal T}\}.
\]
By the symbolic pattern-count formula along the F\o lner sequence $(L_m)$,
tand since$h(X_{\mathcal T},G)=h(\mathcal T)=0$, we
\[
\lim_{m\to\infty}\frac{\log N_m}{|L_m|}=0.
\]
Since $|L_m|/|K_m|\to1$, it follows that
\[
\frac{\log N_m}{|K_m|}
=\frac{|L_m|}{|K_m|}\frac{\log N_m}{|L_m|}
\longrightarrow0.
\]
Therefore, for the previously fixed $\beta>0$, there exists $m_0$
such that
\begin{equation}\label{eq:phase-number-new}
	\#\{p_m(g):g\in G\}\leq e^{\beta|K_m|}
	\qquad(m\geq m_0).
\end{equation}
This estimates all translations simultaneously; in particular, the
threshold $m_0$ is independent of the chosen point $y$ and its
representative $g_y$.
For $y = g_y z_y$, define
\[
\pi^1_m(y) := p_m(g_y).
\]
and $Im(\pi^1_m) := \{\pi^1_m(y) : y\in Y\}$.

For fixed $p_m(g_y)$, record the labels on good subtiles inside complete
tiles of $K_mg_y$ by $(S,d,i,c)$. For $p\in Im(\pi^1_m)$, let $Y(p)= \{y\in Y : \pi^1_m(y) = p\}$. Define
\[
\pi^2_m(y) :=  \prod_{S,d,i,c} \omega_{S,d,i,c} \in  \prod_{S,d,i,c} \Gamma_{{S,i,c}}
\]
Since these subtiles are disjoint, we have
\begin{equation}\label{eq:label-code-upper}
\#\{\pi^2_m(y) : y\in Y(p)\} \leq  \prod_{S,d,i,c}  \#\Gamma_{S,i,c}\leq e^{(h_1+\beta)|K_m|}.
\end{equation}

%\smallskip\emph{Mistake sets.} 
Fix $y = g_yz_y \in Y(p)$. On each good subtile $(S,d,i,c)$, record
\[
M_{S,d,i,c}(y)=\{u\in J_{S,i,c}:
\rho(ucdz_y,u\omega_{S,d,i,c})>\varepsilon\}.
\]
Its size is at most $\delta|J_{S,i,c}|$. Define
\[
\pi^3_m(y) := \prod_{S,d,i,c} M_{S,d,i,c} cd
\]
By the finite-$N$ estimate
\eqref{eq:finite-mistake-count}, 
\begin{equation}\label{eq:mistake-code-upper}
\# \{\pi^3_m(y) : y\in Y(p)\} \leq \prod_{S,d,i,c}e^{\Hb(\delta)|J_{S,i,c}|}
\leq e^{\Hb(\delta)|K_m|}.
\end{equation}
In particular, no asymptotic error is accumulated over the individual blocks.

%\smallskip\emph{Values at uncontrolled positions.} 
Fix $p \in Im(\pi^1_m)$ and $y = g_yz_y \in Y(p)$. Let $\Lambda_{K_m g_y} = \{Sd \in \mathcal{T} : Sd \subset K_mg_y\}$ and $\widetilde{K_mg_y} = \bigcup_{Sd\in \Lambda_{K_m g_y}} Sd$. Set $B_{\rm geo}(p) =K_m \setminus \widetilde{K_mg_y}$, which consists of the coordinates whose translates
in $K_mg_y$ lie outside the good subtiles of complete tiles. Uniformly in
the tiling pattern, we have
\[
|B_{\rm geo}(p)|\leq(\theta+4\kappa)|K_m|.
\]
Let $B_{Mist}(y) = \bigcup\limits_{S,d,i,c} M_{S,d,i,c}(y) c d$. Then $|B_{Mist}(p)| \leq \delta |K_m| $.
Let  $B_m(y)= B_{\rm geo}(p) \cup B_{Mist}(y).$ We have
\[
|B_m(y)|\leq(\theta+4\kappa+\delta)|K_m|.
\]
Recall that $Q$ is a finite $\varepsilon/4$-spanning set with the minimum cardinality and write $Q = \{q_1,\cdots,q_{\# Q}\}$. For each $s\in B_m(y)$, let $l_s := \min \{1\leq l \leq \# : \rho(sy, q_l) < \frac{\varepsilon}{4}\}$ and set 
\[
\pi^4_m(y)(s) := q_{l_s}.
\]
The domain is already determined by the pattern and mistake codes.
Thus the value code for uncontrolled  positions has at most
$|Q|^{(\theta+4\kappa+\delta)|K_m|}$ possibilities. Define
\[
\pi^4_m(y) := \prod_{s\in B_m(y)} \pi^4_m(y)(s) \in Q^{B_m(y)}.
\]
Then we have
\[
\# \{\pi^4_m(y): y\in Y(p)\}\leq \#Q^{(\theta+4\kappa+\delta)|K_m|}.
\]

For $y = g_y z_y$, let \[
\Pi_m(y)=\bigl(\pi^1_m(y)),\pi^2_m(y),\pi^3_m(y),\pi^4_m(y)\bigr)
\]
be the complete code, where the four components record, respectively,
the tiling pattern, the labels on complete good subtiles, the mistake
sets on those subtiles, and the $Q$-values at the uncontrolled sites.
We index the label and mistake components by the relative data
$(S,a,i,c)$, where $a=dg_y^{-1}$ for a complete tile $Sd\subset K_mg_y$.
Thus the corresponding subtile has coordinates $J_{S,i,c}ca\subset K_m$.
This convention allows codes belonging to different translations $g_y$
to be compared on the same observation window $K_m$.

For each fixed pattern, the number of label codes is at most
$e^{(h_1+\beta)|K_m|}$. For each fixed pattern and label code, the
number of mistake codes is at most $e^{\Hb(\delta)|K_m|}$.
Once these components have been fixed, the set of uncontrolled
coordinates $B_m(y)$ is determined, and its cardinality is at most
$(\theta+4\kappa+\delta)|K_m|$. Thus the number of value codes over
each such choice is at most $|Q|^{(\theta+4\kappa+\delta)|K_m|}$.
Multiplying these uniform upper bounds, together with the pattern bound
\eqref{eq:phase-number-new}, gives
\[
\begin{aligned}
	\#\{\Pi_m(y):y\in Y\}
	&\leq e^{\beta|K_m|}e^{(h_1+\beta)|K_m|}
	e^{\Hb(\delta)|K_m|}
	\#Q^{(\theta+4\kappa+\delta)|K_m|}\\
	&=\exp\!\Bigl(\bigl[h_1+2\beta
	+\Hb(\delta)+\delta\log\#Q\\
	&\hspace{42mm}+(\theta+4\kappa)\log \#Q\bigr]|K_m|\Bigr).
\end{aligned}
\]
The two parameter choices \eqref{eq:delta-choice-new} and
\eqref{eq:bad-parameters-new} imply
\[
\Hb(\delta)+\delta\log\#Q<\beta,
\qquad (\theta+4\kappa)\log\#Q<\beta.
\]
Consequently, for all sufficiently large $m$,
\begin{equation}\label{eq:total-code-number-new}
	\#\{\Pi_m(y):y\in Y\}\leq e^{(h_1+4\beta)|K_m|}.
\end{equation}
This multiplication only uses uniform bounds on the number of extensions
of each partial code; independence of the four components is not needed.

Suppose $y_1,y_2\in Y$ satisfy $\Pi_m(y_1)=\Pi_m(y_2)$, and write
$y_j=g_jz_j$ using their fixed representations. Equality of the pattern
components identifies the complete tiles and good subtiles in relative
coordinates. Equality of the mistake components then implies
\[
B_m(y_1)=B_m(y_2)=:B.
\]
We compare the orbit coordinates at each $r\in K_m$.

If $r\in B$, equality of the value components gives a common point
$q_r\in Q$ such that
\[
\rho(ry_1,q_r)<\varepsilon/4,
\qquad \rho(ry_2,q_r)<\varepsilon/4.
\]
Hence the triangle inequality yields
\[
\rho(ry_1,ry_2)<\varepsilon/2.
\]

If $r\in K_m\setminus B$, the common pattern places $r$ in a good
subtile of a complete tile, and the common mistake code says that this
coordinate is not a mistake for either point. More explicitly, there
are common relative indices $(S,a,i,c)$ and a common
$u\in J_{S,i,c}$ such that
\[
r=uca.
\]
The corresponding tiles in the original tiling have centers
$d_j=ag_j$, and their good subtiles are $J_{S,i,c}cd_j$.
Equality of the label components supplies the same label
$w\in\Gamma_{S,i,c}$ on these two corresponding subtiles.
Since $u$ belongs to neither mistake set, the tracing inequalities give
\[
\rho(ucd_jz_j,uw)\leq\varepsilon
\qquad(j=1,2).
\]
Using the left-action convention and $d_j=ag_j$, we obtain
\[
ry_j=(uca)(g_jz_j)=ucd_jz_j.
\]
Thus both orbit coordinates are within $\varepsilon$ of the same
point $uw$, and
\[
\rho(ry_1,ry_2)
\leq\rho(ry_1,uw)+\rho(uw,ry_2)\leq2\varepsilon.
\]
The two cases exhaust $K_m$. Therefore
\[
\Pi_m(y_1)=\Pi_m(y_2)
\quad\Longrightarrow\quad
\rho_{K_m}(y_1,y_2)
=\max_{r\in K_m}\rho(ry_1,ry_2)\leq2\varepsilon.
\]

Let $E\subset Y$ be $(K_m,3\varepsilon)$-separated. Two distinct
points of $E$ cannot have the same code, since their Bowen distance
would then be at most $2\varepsilon$. Hence $\Pi_m|_E$ is injective,
and \eqref{eq:total-code-number-new} gives
\[
|E|\leq\#\{\Pi_m(y):y\in Y\}
\leq e^{(h_1+4\beta)|K_m|}.
\]
Taking the supremum over such sets $E$ yields
\[
s(Y,K_m,3\varepsilon)\leq e^{(h_1+4\beta)|K_m|}.
\]

Finally, let $E\subset\Lambda=\overline Y$ be a finite
$(K_m,5\varepsilon)$-separated set. Since $K_m$ is finite and the
action is continuous, for each $x\in E$ the set
\[
\{y\in\Lambda:\rho_{K_m}(x,y)<\varepsilon/2\}
\]
is a relative open neighborhood of $x$. By density of $Y$ in $\Lambda$,
choose $y_x\in Y$ in this neighborhood. If $x\ne x'$ belong to $E$,
the triangle inequality in the Bowen metric implies
\[
\begin{aligned}
	\rho_{K_m}(y_x,y_{x'})
	&\geq\rho_{K_m}(x,x')-\rho_{K_m}(x,y_x)
	-\rho_{K_m}(x',y_{x'})\\
	&>5\varepsilon-\varepsilon=4\varepsilon>3\varepsilon.
\end{aligned}
\]
Thus the points $y_x$ are distinct and form a
$(K_m,3\varepsilon)$-separated subset of $Y$ with cardinality $|E|$.
It follows that
\[
s(\Lambda,K_m,5\varepsilon)
\leq s(Y,K_m,3\varepsilon)
\leq e^{(h_1+4\beta)|K_m|}.
\]
Taking logarithms, dividing by $|K_m|$, and passing to the upper limit,
we conclude that
\[
h(\Lambda,G,5\varepsilon)\leq h_1+4\beta
=h_0+9\beta<h_0+\beta_0.
\]
Since $\gamma=5\varepsilon<\gamma_0$, this proves item (3).

Let $\Pi_m(y)$ be the resulting complete code. Combining these estimates
with \eqref{eq:delta-choice-new} and \eqref{eq:bad-parameters-new}, we get
\begin{equation}\label{eq:total-code-number-new}
\#\{\Pi_m(y):y\in Y\}\leq e^{(h_1+4\beta)|K_m|}.
\end{equation}
If two points have the same code, their orbit coordinates at every
uncontrolled site are within $\varepsilon/2$, and at every remaining
site they are within $2\varepsilon$, because both trace the same local
coordinate of the same label. Hence
\[
\Pi_m(x)=\Pi_m(y)\quad\Longrightarrow\quad
\rho_{K_m}(x,y)\leq2\varepsilon.
\]
It follows that
$s(Y,K_m,3\varepsilon)\leq e^{(h_1+4\beta)|K_m|}$.

Finally, any $(K_m,5\varepsilon)$-separated set in $\Lambda=\overline Y$
can be moved into $Y$ by less than $\varepsilon/2$ in the Bowen metric,
giving a $(K_m,3\varepsilon)$-separated set of the same cardinality.
Therefore
\[
h(\Lambda,G,5\varepsilon)\leq h_1+4\beta
=h_0+9\beta<h_0+\beta_0.
\]
This proves item (3).
\end{proof}

\begin{remark}\label{rem:upper-code}
Besides ensuring invariance after trimming, the main additional estimate
for approximate tracing occurs in Step 5. Exact tracing has no mistake-set
code. Under $\APP$
the additional normalized entropy is bounded by
\[
\Hb(\delta)+\delta\log|Q|,
\]
which tends to zero with $\delta$.
\end{remark}

By Proposition \ref{mainprop1}, we can generalize Theorem B in \cite{eizenberg1994large} to amenable group actions.

\begin{theorem}\label{continuitythm}
	Suppose that $(X,G)$ has the approximate product property and the entropy map $h_\mu(X,G) K \M(X,G) \mapsto \mathbb{R}$ is upper semicontinuous at the point $\nu \in \M(X,G)$. Then there exists a sequence a $G$-invariant ergodic measures $\{\mu_n\} $ such that
	\[
	\mu_n \overset{w}{\rightarrow} \nu \text{ and } h_{\mu_n}(X,G) \rightarrow h_{\nu}(X,G),
	\]
	where $ \overset{w}{\rightarrow} $ denotes the convergence in the weak topology of $\M(X)$.
\end{theorem} 

\begin{proof}
	By Proposition \ref{mainprop1}, there is a sequence of $G$-invariant ergodic measures $\{\mu_n\}$ so that
	\[
	\mu_n  \overset{w}{\rightarrow} \nu \text{ and } \liminf\limits_{n\rightarrow \infty} h_{\mu_n}(X,G)\geq h_{\nu}(X,G).
	\]
	By the upper semi-continuity of the entropy map $h_\mu$, we have
	\[
	\limsup\limits_{n\rightarrow \infty} h_{\mu_n}(X,G)\leq h_{\nu}(X,G).
	\]
	Thus we have
	\[
	\mu_n  \overset{w}{\rightarrow} \nu \text{ and } \lim\limits_{n\rightarrow \infty} h_{\mu_n}(X,G) = h_{\nu}(X,G).
	\]
\end{proof}
\begin{proposition}\label{mainpropositon16r}
Assume that $(X,G)$ has the approximate product property and let
$\mu_0\in\M_e(X,G)$ with $h_{\mu_0}(X,G)=0$.  For every $\eta_0>0$ there
exists a nonempty compact $G$-invariant set $\Lambda\subset X$ such that
\[
\M(\Lambda,G)\subset \mathcal{B }(\mu_0,\eta_0).
\]
\end{proposition}

\begin{proof}
Choose a point $x^*$ generic for $\mu_0$ along the auxiliary tempered
F\o lner sequence $(F_n)$. Fix small $\eta,\varepsilon,\delta$,
then obtain the APP data and choose $\kappa,\theta$ as in
\eqref{eq:main-trim-choice}--\eqref{eq:bad-parameters-new}, omitting
the separation and entropy conditions.  In the construction of Proposition~\ref{mainprop1}, use the
single label $x^*$ on every sufficiently large good block instead of the
sets $\Gamma_{S,i,c}$.  The APP gluing and the empirical-measure estimate in
Step 3 remain valid, while no entropy estimate is needed.  Taking
$\Lambda=\overline{GZ^\#}$ gives the conclusion.
\end{proof}

\subsection{Proof of Theorem~\ref{mainthm2}}

\begin{proof}[Proof of Theorem~\ref{mainthm2}]
Let $\mu\in\M(X,G)$ and let $U$ be a neighborhood of $\mu$.
If $h_\mu(X,G)=0$, choose $\eta_0>0$ with $B(\mu,2\eta_0)\subset U$.
By Proposition~\ref{propositionentropydenseb}, $\M_e(X,G)$ is dense in
$\M(X,G)$, so choose $\mu_0\in\M_e(X,G)$ with
$D(\mu,\mu_0)<\eta_0$.  Apply Proposition~\ref{mainprop1} if
$h_{\mu_0}>0$, and Proposition~\ref{mainpropositon16r} if
$h_{\mu_0}=0$.  In either case there is a nonempty compact invariant
$\Lambda$ with
\[
\M(\Lambda,G)\subset B(\mu_0,\eta_0)
\subset B(\mu,2\eta_0)\subset U.
\]

Now assume $h_\mu(X,G)>0$.  Fix
$h\in(0,h_\mu(X,G))$ and $\varepsilon,\beta>0$, and choose
$\eta_0>0$ with $B(\mu,2\eta_0)\subset U$.  Entropy-denseness gives
$\mu_0\in\M_e(X,G)$ satisfying
\[
D(\mu,\mu_0)<\eta_0,
\qquad h_{\mu_0}(X,G)>h.
\]
Apply Proposition~\ref{mainprop1} to
$\mu_0,h,\eta_0,\beta,\varepsilon$.  We obtain a compact invariant set
$\Lambda$ and $\gamma\in(0,\varepsilon)$ such that
\[
\M(\Lambda,G)\subset B(\mu_0,\eta_0),
\qquad h(\Lambda,G)>h,
\qquad h(\Lambda,G,\gamma)<h+\beta.
\]
The triangle inequality gives $\M(\Lambda,G)\subset U$, proving almost
entropy-approximability.
\end{proof}

\section{Entropy-approximability and entropy-genericity}\label{sec:exact}

In this section we upgrade almost entropy-approximability to
entropy-approximability under asymptotic entropy expansiveness. The argument is the
amenable-group analogue of the upgrading step in \cite{sun2025ergodic}. The only
additional input needed here is the local-entropy estimate recalled in
Proposition~\ref{asympexpansive2}.

\begin{proposition}\label{prop:almost-to-entropy}
Assume that $(X,G)$ is asymptotically entropy expansive. If
$\mu\in\mathcal M(X,G)$ is almost entropy-approximable, then it is
entropy-approximable.
\end{proposition}

\begin{proof}
If $h_\mu(X,G)=0$, the assertion follows from the convention in
Definition~\ref{def:entropyapprox}. Assume that $h_\mu(X,G)>0$. Let $U$ be a
weak$^*$ neighborhood of $\mu$, let $0<h<h_\mu(X,G)$, and let $\beta>0$.
Choose $\eta>0$ so small that $2\eta<\beta$.

Since $(X,G)$ is asymptotically entropy expansive, there exists
$\varepsilon_0>0$ such that
\begin{equation}\label{eq:aee-small-tail}
h_{\mathrm{loc}}^X(\gamma,G)<\eta
\qquad\text{for every }0<\gamma<\varepsilon_0.
\end{equation}
Apply almost entropy-approximability of $\mu$ with neighborhood $U$, target entropy
$h$, scale bound $\varepsilon_0$, and error $\eta$. We can find a nonempty compact
$G$-invariant set $\Lambda\subset X$ and some $\gamma\in(0,\varepsilon_0)$ such that
\begin{equation}\label{eq:aee-almost-appr}
\mathcal M(\Lambda,G)\subset U,
\qquad h(\Lambda,G)>h,
\qquad h(\Lambda,G,\gamma)<h+\eta.
\end{equation}
By Proposition~\ref{asympexpansive2},
\[
h(\Lambda,G)
\le h(\Lambda,G,\gamma)+h_{\mathrm{loc}}^X(\gamma,G)
<h+2\eta<h+\beta.
\]
Together with \eqref{eq:aee-almost-appr}, this proves
\[
\mathcal M(\Lambda,G)\subset U,
\qquad h<h(\Lambda,G)<h+\beta,
\]
so $\mu$ is entropy-approximable.
\end{proof}

\begin{remark}\label{rem:aee-upgrading}
The role of asymptotic entropy expansiveness in the preceding proposition is to make
the entropy hidden below sufficiently small Bowen scales uniformly negligible. No
separate upper-semicontinuity hypothesis on the entropy map is imposed.
\end{remark}

\begin{proposition}\label{prop:entropy-generic-abstract}
Suppose that $(X,G)$ is asymptotically entropy expansive and every measure in
$\mathcal M(X,G)$ is entropy-approximable. Then $(X,G)$ is entropy-generic.
\end{proposition}

\begin{proof}
Fix $0\le\alpha<h(X,G)$. By \cite{ren2016local}, asymptotic entropy expansiveness
implies upper semicontinuity of the entropy map. Hence
\[
\mathcal M^\alpha(X,G)=\{\mu:h_\mu(X,G)\ge\alpha\}
\]
is compact and is therefore a Baire space. For $\alpha'>\alpha$ let
\[
\mathcal M_e(\alpha,\alpha')
:=\{\nu\in\mathcal M_e(X,G):\alpha\le h_\nu(X,G)<\alpha'\}.
\]
The ergodic measures form a $G_\delta$ subset of $\mathcal M(X,G)$, and upper
semicontinuity makes $\{\nu:h_\nu(X,G)<\alpha'\}$ open. Thus
$\mathcal M_e(\alpha,\alpha')$ is a $G_\delta$ subset of
$\mathcal M^\alpha(X,G)$.

We prove density. Let $\mu\in\mathcal M^\alpha(X,G)$ and let $V$ be a neighborhood of
$\mu$. Take $\eta>0$ small such that $\mathcal{B}(\mu,2\eta) \subset V$. Since $\alpha<h(X,G)$, the variational principle gives
an invariant measure $\mu_X$ with $h_{\mu_X}(X,G)>\alpha$. For sufficiently small
$t>0$,
\[
\mu'=(1-t)\mu+t\mu_X
\]
belongs to $\mathcal{B}(\mu,2\eta)$, and affinity of entropy gives $\alpha < h_{\mu'}(X,G) < \alpha'$. By Theorem \ref{continuitythm}, there is an ergodic measure $\nu$ satisfying
\[
D(\nu, \mu') < \eta \text{ and } h_{\nu}(X,G) \in (\alpha, \alpha').
\]
Hence we have proved density.

Therefore $\mathcal M_e(\alpha,\alpha')$ is residual in
$\mathcal M^\alpha(X,G)$. Taking $\alpha'=\alpha+1/n$ and intersecting over $n\ge1$
gives
\[
\mathcal M_e(X,G,\alpha)
=\bigcap_{n=1}^{\infty}\mathcal M_e(\alpha,\alpha+1/n),
\]
which is residual in $\mathcal M^\alpha(X,G)$.
\end{proof}

\begin{proof}[Proof of Theorem~\ref{mainthm3}]
By Theorem~\ref{mainthm2}, every invariant measure is almost entropy-approximable.
Proposition~\ref{prop:almost-to-entropy} upgrades this to entropy-approximability under
asymptotic entropy expansiveness. Proposition~\ref{prop:entropy-generic-abstract} then
gives entropy-genericity.
\end{proof}

\begin{proof}[Proof of Theorem~\ref{mainthm1}]
Fix an entropy value $0\le\alpha<h(X,G)$.
Theorem~\ref{mainthm3} says that the corresponding ergodic entropy level set
is residual in the nonempty compact space $\mathcal M^\alpha(X,G)$.
It is therefore nonempty. Thus every entropy value strictly
below $h(X,G)$ is realized by an ergodic invariant measure.

Asymptotic entropy expansiveness implies upper semicontinuity of the entropy map on the
compact space $\mathcal M(X,G)$. Hence the entropy map attains its maximum. By the
variational principle this maximum equals $h(X,G)$. If a maximizing invariant measure
is not ergodic, affinity of entropy and ergodic decomposition imply that almost every
ergodic component is maximizing. Therefore the endpoint $h(X,G)$ is also realized by
an ergodic measure.
\end{proof}

\appendix
\section{Entropy-denseness under the approximate product property}\label{sec:appendix}
We prove Proposition~\ref{propositionentropydenseb} independently of
Theorem~\ref{mainthm2}. Only the geometric construction and entropy lower
bound for independently chosen labels are needed.

\begin{lemma}\label{lem:balanced-coloring}
In Lemma~\ref{lem:admissible-subtiles}, the final shapes can be chosen so
that their good subtiles admit colors $1,\dots,L$ with
\begin{equation}\label{eq:color-weights}
\left|w_i(S)-\frac1L\right|<5\kappa,
\qquad
w_i(S):=\frac1{|S|}\sum_{Jc\text{ of color }i}|J|.
\end{equation}
The uncolored remainder has proportion less than $4\kappa$.
\end{lemma}
\begin{proof}

By Lemma~\ref{lem:admissible-subtiles}, we can require the largest subtile to have size less than $\kappa|S|$ in every
shape $S\in \mathcal{S}$. Assign the
subtiles successively to a color of smallest current total size. At every
stage, the largest and smallest totals differ by at most the largest
subtile size. If $W$ is the final covered size, then
$|w_i(S)-W/(L|S|)|<\kappa$. Since
$(1-4\kappa)|S|<W\leq|S|$, \eqref{eq:color-weights} follows.
\end{proof}

\begin{proof}[Proof of Proposition~\ref{propositionentropydenseb}]
Fix $\mu\in\M(X,G)$, $\eta_0>0$ and a finite $h_0<h_\mu(X,G)$.
Using the ergodic decomposition and the fact that the entropy is an affine function on $\M(X,G)$,
 it follows that with no loss of generality we may assume $\mu$ is a finite sum of ergodic measures $\mu_1,\cdots,\mu_L$, i.e.
 \[
 \mu = \frac{1}{L}\sum_{l=1}^{L}\mu_l,
 \]
 where the $\mu_l$'s are not necessarily distinct. We will finish the proof under this assumption. 
Choose finite $a_l\geq0$, strictly below $h_{\mu_l}(X,G)$ whenever
$h_{\mu_l}(X,G)>0$, and equal to zero otherwise, so that
\[
A:=\frac1L\sum_l a_l>h_0.
\]
Put $L_+=\{l:a_l>0\}$. For $l\in L_+$ apply
Proposition~\ref{mainseppro1} with target $a_l$, obtaining
$\delta_l^*,\varepsilon_l^*>0$. For $l\notin L_+$ choose one generic
point $x_l$ for $\mu_l$ along the F{\o}lner sequence $(F_n)$.

Choose $b>0$ with $2b<\eta_0/8$, then choose $0<\Delta,\delta<1$
so that $\Delta+\delta<\eta_0/8$ and
\[
\Delta<\frac14\min_{l\in L_+}\varepsilon_l^*,
\qquad 2\delta<\frac12\min_{l\in L_+}\delta_l^*.
\]
Obtain $m(\delta,\Delta)=(K_{\rm APP},\eta_{\rm APP})$ and only then
choose $\kappa,\theta>0$ satisfying
\begin{gather}
3\kappa<b,\quad 5L\kappa<1,\quad
\frac{3|K_{\rm APP}|\kappa}{1-3\kappa}<\frac{\eta_{\rm APP}}2,
\label{eq:appendix-small}\\
(5L+4)\kappa+\theta<\eta_0/4,
\qquad (1-\theta)(1-5L\kappa)A>h_0.
\label{eq:appendix-entropy-budget}
\end{gather}
When $L_+\ne\varnothing$, we also require
$3\kappa<\frac12\min_{l\in L_+}\delta_l^*$.
If $A=0$, then $h_0<0$, and the last inequality is automatic.

For color $l$, define
\[
\T_{l}(S) := \{ Jc ; Jc \text{ is of color } l\}.
\]
and $w_l(S) := \frac{1}{S} \sum\limits_{Jc \text{ of color } l}|J|$.
By Lemma \ref{lem:balanced-coloring}, we have
\[
\big|w_l(S)- \frac{1}{L}\big| < 5\kappa 
\]
and the uncolored remainder has proportion less than $4\kappa$.
Also we take the block tail
so far out that all generic-label conclusions for the finitely many
measures hold in $\B(\mu_l,b)$. By Lemma~\ref{eslemma1}, for a good subtile $J$ of color
$l\in L_+$, we can choose a label set $\Gamma_J^{(l)}$ with
\[
|\Gamma_J^{(l)}|\geq e^{a_l|J|},
\]
which is $(\delta_l^*/2,J,\varepsilon_l^*)$-separated and whose empirical
measures lie in $\B(\mu_l,2b)$.
For $l\notin L_+$, let $\Gamma_J^{(l)} =\{x_l\}$ and by Lemma~\ref{measures2}, we also have 
$
\mathcal{E}_{J}(x_l) \in \mathcal{B}(\mu_l, 2b).
$

Fix the coloring and the
label sets once for each shape, and translate this data to every tile.

Index the tiles by $Sd\in\mathcal T$ and set
\[
\Omega=\prod_{Sd\in\mathcal T}\prod_{i=1}^t
\prod_{c\in C_{S,i}}\Gamma_{S,i,c}.
\]
For $\omega = (\omega_{S,d,i,c})\in\Omega$, $\omega_{S,d,i,c} \in \Gamma^{l}_{J_{S,i,c}}$ if $J_{S,i,c} c\in \T_j(S)$. Apply $\APP$ on the pairwise disjoint blocks
$J_{S,i,c}cd$ to the points $d^{-1}c^{-1}\omega_{S,d,i,c}$.
Right invariance \eqref{eq:right-boundary} makes these blocks admissible.
Using the approximate product property, the set
\[
Z^\#(\omega)=\left\{z:
\begin{array}{l}
	|\{u\in J_{S,i,c}:\rho(ucdz,u\omega_{S,d,i,c})>\varepsilon\}|\\
	\hfill{}\leq\delta|J_{S,i,c}|\quad\text{for every }S,d,i,c
\end{array}\right\}
\]
is nonempty. The relation $R=\{(z,\omega):z\in Z^\#(\omega)\}$
is closed in the compact space $X\times\Omega$, since each block
condition is a finite union of intersections of closed conditions.
Consequently $Z^\#=\pi_X(R)$ is nonempty and compact. Put
\[
Y= \bigcup_{g\in G} g Z^\# = GZ^\#,\text{ and } \Lambda=\overline Y.
\]
Then $\Lambda$ is a nonempty, compact and $G$-invariant subset of $X$.
On each good subtile of color $l$, the empirical measure of a tracing
point is within
$e_0:=2b+\Delta+\delta<\eta_0/4$ of $\mu_l$.
We spell out the error over a complete tile, since the coloring is by
area and its error depends on $L$.

Fix a tile $Sd$. Its color weights satisfy \eqref{eq:color-weights},
and let $r_S=1-\sum_{i=1}^{L}w_i(S)<4\kappa$.

Recall that for $\lambda,\lambda' \in \M(X)$, 
\[
D(\lambda,\lambda')
=
\sum_{j=1}^{\infty}2^{-j}
\left|
\int\varphi_j\,d\lambda-\int\varphi_j\,d\lambda'
\right|,
\qquad 0\leq\varphi_j\leq1.
\]
For each test function \(\varphi_j\), we compare its average over \(Sd\) with its integral against \(\mu\). There are three contributions to the error. The first is the tracing error on the prescribed orbit pieces. The second comes from replacing the actual color weights \(w_l(S)\) by the target weights \(1/L\). Since the averages of \(\varphi_j\) associated with each color belong to \([0,1]\), this contribution is at most
\[
\sum_{i=1}^{L}\left|w_i(S)-\frac1L\right|.
\]
The third contribution comes from the remainder of the tile and is at most its relative size \(r_S\), again because \(0\leq\varphi_j\leq1\).

Writing \(\tau_j\) for the tracing-error contribution for \(\varphi_j\), we therefore obtain
\[
\left|
\int\varphi_j\,d\mathcal E(z,Sd)
-
\int\varphi_j\,d\mu
\right|
\leq
\tau_j+
\sum_{i=1}^{L}\left|w_i(S)-\frac1L\right|+r_S.
\]
The preceding tracing estimate gives
\(\sum_{j=1}^{\infty}2^{-j}\tau_j\leq e_0\).
By the choice of metric $D$, we conclude that
\[
\begin{aligned}
	D\bigl(\mathcal E(z,Sd),\mu\bigr)
	&\leq
	e_0+\sum_{i=1}^{L}\left|w_i(S)-\frac1L\right|+r_S\\
	&<e_0+(5L+4)\kappa.
\end{aligned}
\]
This bound holds uniformly over all complete tiles.

We next pass from individual tiles to \(K_mg\), where $g\in G$. Let $\Lambda_{K_m g} = \{Sd\in \T : Sd \subset Kmg\}$.  Write $\Lambda_{K_m g}  = \{T_1,\ldots,T_N\}$
 and put
\[
U_{m,g}=\bigcup_{\ell=1}^{N}T_\ell,
\qquad
R_{m,g}=K_mg\setminus U_{m,g}.
\]
For all sufficiently large \(m\), uniformly in \(g\), the complete-tile coverage estimate implies
\[
\frac{|R_{m,g}|}{|K_m|}<\theta.
\]
Since the tiles are disjoint, the empirical measure on \(K_mg\) decomposes as
\[
\mathcal E(z,K_mg)
=
\sum_{\ell=1}^{N}
\frac{|T_\ell|}{|K_m|}
\,\mathcal E(z,T_\ell)
+
\frac{|R_{m,g}|}{|K_m|}\,\mathcal E(z,R_{m,g}).
\]

Moreover, \(D(\lambda,\lambda')\leq1\) for any two probability measures, because every \(\varphi_j\) takes values in \([0,1]\). Applying these observations to the preceding decomposition yields
\[
\begin{aligned}
	D\bigl(\mathcal E(z,K_mg),\mu\bigr)
	&\leq
	\sum_{\ell=1}^{N}\frac{|T_\ell|}{|K_m|}
	D\bigl(\mathcal E(z,T_\ell),\mu\bigr)
	+\frac{|R_{m,g}|}{|K_m|}\\
	&\leq
	(1-\frac{|R_{m,g}|}{|K_m|})\bigl(e_0+(5L+4)\kappa\bigr)
	+\frac{|R_{m,g}|}{|K_m|}\\
	&< e_0+(5L+4)\kappa+\theta
	<\frac{\eta_0}{2}.
\end{aligned}
\]
Here the last inequality follows from the choice of parameters. In particular, the estimate holds for all sufficiently large \(m\), with the threshold independent of \(g\).

By Lemma~\ref{lem:complete-local}, applied to $Z^\#$, this uniform
estimate implies
\begin{equation}\label{eq:appendix-measure-localization}
D(\nu,\mu)\leq\eta_0/2<\eta_0
\quad\text{for every }\nu\in\M(\Lambda,G).
\end{equation}

If $h_0<0$, any ergodic invariant measure on the nonempty compact system
$\Lambda$ already has entropy greater than $h_0$, so localization completes
the proof in this case. Suppose now that $h_0\geq0$. Then $A>0$ and
$L_+\ne\varnothing$. Vary the labels independently on good subtiles in
complete tiles of $K_m$. Distinct choices differ on a color in $L_+$,
because all other label sets are singletons. The same separation argument
as in Step 4 of Proposition~\ref{mainprop1} gives a common separation
scale, for instance $\frac12\min_{i\in L_+}\varepsilon_i^*$.
The logarithm of the number of labels on a single complete tile is at least
\[
\begin{split}
\sum_i a_i w_i(S)|S|
&\geq\left(\frac1L-5\kappa\right)\sum_i a_i|S|\\
&=(1-5L\kappa)A|S|.
\end{split}
\]
Consequently,
\[
h(\Lambda,G)\geq(1-\theta)(1-5L\kappa)A>h_0.
\]
The variational principle and ergodic decomposition yield
$\nu\in\M_e(\Lambda,G)$ with $h_\nu(X,G)>h_0$.
Together with \eqref{eq:appendix-measure-localization}, this proves
entropy-denseness.
\end{proof}

\noindent{\bfseries Acknowledgement:} We would like to thank Professor Hanfeng Li for his constructive suggestions. Ren is supported by the Fundamental Research Funds for the Central Universities Project No.2025CDJZKPT-09 and No.2026CDJZKPT-09. Zhang is supported by the Natural Science Foundation of Chongqing, China, with grant No.\ CSTB2025NSCQ-GPX0875.
\bibliographystyle{amsplain-initials}
  \bibliography{reference_AEE_revised}

\end{document}